\documentclass[english]{article}
\usepackage[T1]{fontenc}
\usepackage[latin9]{inputenc}
\usepackage[a4paper]{geometry}
\usepackage{color}
\usepackage{babel}
\usepackage{amsthm}
\usepackage{amsmath}
\usepackage{amssymb}
\usepackage[all]{xy}
\usepackage[unicode=true,pdfusetitle,
 bookmarks=true,bookmarksnumbered=false,bookmarksopen=false,
 breaklinks=false,pdfborder={0 0 1},backref=false,colorlinks=false]
 {hyperref}

\makeatletter
\numberwithin{equation}{section}
\numberwithin{figure}{section}
\newcommand{\lyxaddress}[1]{
\par {\raggedright #1
\vspace{1.4em}
\noindent\par}
}
\theoremstyle{plain}
\newtheorem{thm}{\protect\theoremname}
  \theoremstyle{plain}
  \newtheorem{lem}[thm]{\protect\lemmaname}
  \theoremstyle{remark}
  \newtheorem{rem}[thm]{\protect\remarkname}
  \theoremstyle{plain}
  \newtheorem{cor}[thm]{\protect\corollaryname}
  \theoremstyle{plain}
  \newtheorem{prop}[thm]{\protect\propositionname}
  \theoremstyle{definition}
  \newtheorem{example}[thm]{\protect\examplename}

\usepackage{arydshln,mathtools}
\def\Big#1{\makebox(-5,5){\Large$#1$}}

\definecolor{lightgray}{gray}{.60}

\makeatother

  \providecommand{\corollaryname}{Corollary}
  \providecommand{\examplename}{Example}
  \providecommand{\lemmaname}{Lemma}
  \providecommand{\remarkname}{Remark}
\providecommand{\theoremname}{Theorem}
\providecommand{\propositionname}{Proposition}

\begin{document}
\global\long\def\diag{\text{{\rm diag}}}
\global\long\def\Mod{\text{{\rm mod-}}}
\global\long\def\Rep{\text{{\rm rep-}}}
\global\long\def\im{\text{{\rm Im}}}
\global\long\def\dprime{\prime\prime}
\global\long\def\Ext{\text{{\rm Ext}}}
\global\long\def\Dim{\underline{\dim}}
\global\long\def\Hom{\text{{\rm Hom}}}
\global\long\def\rank{\text{{\rm rank}}}
\global\long\def\End{\text{{\rm End}}}

\title{Preinjective subfactors of preinjective Kronecker modules}

\author{Csaba Szántó, István Szöll\H{o}si}

\maketitle

\lyxaddress{Faculty of Mathematics and Computer Science, Babe\c{s}-Bolyai University}

\lyxaddress{str. M. Kog\u{a}lniceanu, nr. 1, 400084, Cluj-Napoca, Romania}
\begin{abstract}
Using a representation theoretical approach we give an explicit numerical
characterization in terms of Kronecker invariants of the subfactor
relation between two preinjective (and dually preprojective) Kronecker
modules, describing explicitly a so called linking module as well. Preinjective
(preprojective) Kronecker modules correspond to matrix pencils having
only minimal indices for columns (respectively for rows). Thus our
result gives a solution to the subpencil problem in these cases (including
the completion), moreover the required computations are straightforward
and can be carried out easily (both for checking the subpencil relation
and constructing the completion pencils based on the linking module).
We showcase our method by carrying out the computations on an explicit
example.
\end{abstract}

\section{Introduction}

Let $K:\xymatrix{1 & 2\ar@<1ex>[l]^{\beta}\ar@<-1ex>[l]_{\alpha}} $
be the \emph{Kronecker quiver}, i.e. the quiver having two vertices
and two parallel arrows and $k$ an arbitrary field. The path
algebra of the Kronecker quiver is the \emph{Kronecker algebra} and
we will denote it by $kK$. A finite dimensional right module over
the Kronecker algebra is called a \emph{Kronecker module}. We denote
by $\Mod kK$ the category of finite dimensional right modules over
the Kronecker algebra. For a module $M\in\Mod kK$, $[M]$ will denote
the isomorphism class of $M$ and $tM:=M\oplus\dots\oplus M$
($t$-times).

A (finite dimensional) $k$-linear representation of the quiver
$K$ is a quadruple $M=(V_{1},V_{2};\varphi_{\alpha},\varphi_{\beta})$
where $V_{1},V_{2}$ are finite dimensional $k$-vector spaces
(corresponding to the vertices) and $\varphi_{\alpha},\varphi_{\beta}:V_{2}\to V_{1}$
are $k$-linear maps (corresponding to the arrows). Thus a $k$-linear
representation of $K$ associates vector spaces to the vertices and
compatible $k$-linear functions (or equivalently, matrices)
to the arrows. A morphism between two representations $M=(V_{1},V_{2};\varphi_{\alpha},\varphi_{\beta})$
and $M^{\prime}=(V_{1}^{\prime},V_{2}^{\prime};\varphi_{\alpha}^{\prime},\varphi_{\beta}^{\prime})$
is a pair of linear maps $(f_{1},f_{2})$, where $f_{1}:V_{1}\to V_{1}^{\prime}$,
$f_{2}:V_{2}\to V_{2}^{\prime}$ and $f_{1}\varphi_{\alpha}=\varphi_{\alpha}^{\prime}f_{2}$,
$f_{1}\varphi_{\beta}=\varphi_{\beta}^{\prime}f_{2}$. Let us denote
by $\Rep kK$ the category of finite dimensional $k$-representations
of the Kronecker quiver. There is a well-known equivalence of categories
between $\Mod kK$ and $\Rep kK$, so that every Kronecker
module can be identified with a representation of $K$.

The indecomposable Kronecker modules are of three types: preprojective, preinjective and regular (see the details in the next section). A Kronecker module is preinjective (preprojective) if all its indecomposable components are preinjective (preprojective).

It is easy to see that Kronecker modules are uniquely determined up to isomorphism by two sequences of nonnegative integers and some partitions (see the details in the next section). These numerical invariants are called \emph{the Kronecker invariants of the module}.

Recall that a module $M^{\prime}$ is a \emph{subfactor} of $M$ if there
exists a module $L$ with a monomorphism $L\to M$ and an epimorphism
$L\to M^{\prime}$ (or equivalently with an epimorphism $M\to L$
and a monomorphism $M^{\prime}\to L$). We will call $L$ a \emph{linking
module}.

We have given in \cite{SzSz} a numerical criterion in terms of Kronecker
invariants for the existence of a monomorphism between two
preinjective Kronecker modules (and dually for the existence of an
epimorphism between two preprojective modules). The criterion is
very simple, it is in fact a weighted dominance relation between the
invariants. Using this criterion one can also obtain the numerical
description of an epimorphism between two preinjective Kronecker
modules (and dually of a monomorphism between two
preprojective modules) (see \cite{Szo3}). The approach used to
obtain the results above is representation theoretical, the methods
are homological combined with the knowledge on the category
$\Mod kK$.

Note that a different criterion was given by Han Yang in \cite{Han},
working over an algebraically closed field. He uses calculation of
ranks of matrices over polynomial rings and a so called
generalization and specialization approach. These matrices appear in
the representations and in the morphisms between representations.

In this paper we give a simple explicit numerical characterization
in terms of Kronecker invariants of the subfactor relation between
two preinjective (and dually preprojective) Kronecker modules and
also describe an explicit linking module.

\medskip

Regarding the Kronecker modules as representations it is obvious
that a Kronecker module corresponds to a pair of matrices of the
same dimension, thus defining a matrix pencil. Some papers dealing
with this connection are the following: \cite{Bong,BeBo,Han,Sza5}.

Recall that a \textit{matrix pencil} over a field $k$ is a matrix
$A+\lambda B$ where $A,B$ are matrices over $k$ of the same size and
$\lambda$ is an indeterminate. Two pencils $A+\lambda B$,
$A^{\prime}+\lambda B^{\prime}$ are \textit{strictly equivalent},
denoted by $A+\lambda B\sim A^{\prime}+\lambda B^{\prime}$, if and
only if there exists invertible, constant ($\lambda$ independent)
matrices $P,Q$ such that $P(A^{\prime}+\lambda
B^{\prime})Q=A+\lambda B$.

Every matrix pencil is strictly equivalent to a canonical diagonal
form, described by the \textit{classical Kronecker invariants}, namely
\textit{the minimal indices for columns, the minimal indices for rows,
the finite elementary divisors }\textit{\emph{and}}\textit{ the infinite
elementary divisors} (see \cite{Gant} for all the details and \cite{Sza5}
for a worked out example).

A pencil $A^{\prime}+\lambda B^{\prime}$ is called \textit{subpencil}
of $A+\lambda B$ if and only if there are pencils $A_{12}+\lambda B_{12}$,
$A_{21}+\lambda B_{21}$, $A_{22}+\lambda B_{22}$ such that
\[
A+\lambda B\sim\left(\begin{array}{cc}
A^{\prime}+\lambda B^{\prime} & A_{12}+\lambda B_{12}\\
A_{21}+\lambda B_{21} & A_{22}+\lambda B_{22}
\end{array}\right).
\]

In this case we also say that the subpencil can be completed to the
bigger pencil. We speak about row completion when
$A_{12},B_{12},A_{22},B_{22}$ are zero matrices and about column
completion when $A_{21},B_{21},A_{22},B_{22}$ are zero.

There is an unsolved challenge in pencil theory with lots of applications
in control theory (problems related to pole placement, non-regular
feedback, dynamic feedback etc. may be formulated in terms of matrix
pencils, for details see \cite{Loi}).

\textbf{Challenge:} If $A+\lambda B$, $A^{\prime}+\lambda B^{\prime}$
are pencils over $\mathbb{C}$, find a necessary and sufficient condition
in terms of their classical Kronecker invariants for $A^{\prime}+\lambda B^{\prime}$
to be a subpencil of $A+\lambda B$. Also construct the completion
pencils $A_{12}+\lambda B_{12}$, $A_{21}+\lambda B_{21}$, $A_{22}+\lambda B_{22}$.
A particular case of the challenge above is when we limit ourselves
to column or row completions.

Han Yang was the first to give a representation theoretical modular
approach to the matrix subpencil problem. He proved that the
subpencil notion corresponds to the subfactor notion on modular
level. Also the Kronecker invariants of a module correspond to the classical Kronecker invariants of the associated pencil. 

Preinjective Kronecker modules correspond to matrix pencils having
only minimal indices for columns (see details in Section
\ref{sec:Kronecker-modules}). This means that our criterion for
monomorphisms between preinjectives in the particular case when the
factor is of the form $tI_0$ (where $I_0$ is the injective simple module)
coincides with the criteria of Baraga\~na-Zaballa \cite{Bar} and
Mondi\'e \cite{Mon} (see also \cite{Cab},\cite{Dod} as a particular
case) for completing by columns a pencil to another one, both of
them having only minimal indices for columns. One should note that
the existence of a monomorphism in general between two preinjective modules
does not have a natural correspondence in pencil theory.

Having in mind all above one can see that our result on the
subfactor relation between two preinjective Kronecker gives a
solution to the subpencil problem in case both pencils have only
minimal indices for columns. The numerical criterion is simple and
the explicitly described linking module corresponds in fact to a
pencil which is obtained from the smaller one by column completion,
the bigger pencil being a row completion of the linking one. Using
the linking pencil one can also easily construct the completion
pencils. We showcase our result by carrying out the computations on
an explicit example.

Dodig and Sto\u si\'c describe in \cite{DoSt} a numerical criterion
in terms of Kronecker invariants for a pencil having only minimal
indices for columns to be a subpencil of a general one. Later Dodig
gives in \cite{Dodi} a numerical criterion for a general pencil to
be the subpencil of a pencil having only minimal indices for
columns. One can see that the case considered by us is a particular
case of these results, however our criterion is completely different
from the one deducible from the papers above. It also seems very
difficult to see the equivalence of the two criteria.

\section{The category of Kronecker modules\label{sec:Kronecker-modules}}

In this section we give some details on the category of Kronecker
modules (regarded as representations). The interested reader should
consult seminal works such as \cite{Ass,Sim,Aus,Rin} for further
details, proofs and explanations.

For a module $M\in\Mod kK$. For two modules $M,M^{\prime}\in\Mod kK$
will denote by $M^{\prime}\hookrightarrow M$ the fact that there is
a monomorphism $M^{\prime}\to M$) and by $M\twoheadrightarrow
M^{\prime}$ the fact that there is an epimorphism $M\to
M^{\prime}$). Thus, $M^{\prime}$ is a subfactor of $M$ if there
exists a linking module $L$ such that $M^{\prime}\hookrightarrow
L\twoheadleftarrow M$ (or equivalently $M^{\prime}\twoheadleftarrow
L\hookrightarrow M$).

The \emph{simple Kronecker module}s (up to isomorphism) are
\begin{eqnarray*}
S_{1}:k\leftleftarrows0 & \mbox{and} & S_{2}:0\leftleftarrows k.
\end{eqnarray*}
 For a Kronecker module $M$ we denote by $\Dim M$ its \emph{dimension}.
The dimension of $M$ is a vector $\Dim
M=(m_{S_{1}}(M),m_{S_{2}}(M))$, where $m_{S_{i}}(M)$ is the number
of factors isomorphic with the simple module $S_{i}$ in a
composition series of $M$, $i=\overline{1,2}$. Regarded as a
representation,
$M:V_{1}\overset{\varphi_{\alpha}}{\underset{\varphi_{\beta}}{\leftleftarrows}}V_{2}$,
we have that $\Dim M=(\dim_{ k}V_{1},\dim_{ k}V_{2})$.

The \emph{defect} of $M\in\Mod k K$ with $\Dim M=(a,b)$ is defined
in the Kronecker case as $\partial M=b-a$.

An indecomposable module $M\in\Mod k K$ is a member in one of the
following three families: preprojectives, regulars and
preinjectives. In what follows we give some details on these
families.

The \emph{preprojective indecomposable Kronecker module}s are
determined up to isomorphism by their dimension vector. For
$n\in\mathbb{N}$ we will denote by $P_{n}$ the indecomposable
preprojective module of dimension $(n+1,n)$. So $P_{0}$ and $P_{1}$
are the projective indecomposable modules ($P_{0}=S_{1}$ being
simple). It is known that (up to isomorphism) $P_{n}=( k^{n+1},
k^{n};f,g)$, where choosing the canonical basis in $ k^{n}$ and $
k^{n+1}$, the matrix of $f: k^{n}\to k^{n+1}$ (respectively of $g:
k^{n}\to k^{n+1}$)
is $\begin{pmatrix}\mathbb I_{n}\\
0
\end{pmatrix}$ (respectively $\begin{pmatrix}0\\
\mathbb I_{n}
\end{pmatrix}$). Thus in this case
\[
P_{n}:\xymatrix{ k^{n+1} & k^{n}\ar@<1ex>[l]^{\left({\mathbb
I_{n}\atop 0}\right)}\ar@<-0.8ex>[l]_{\left({0\atop \mathbb
I_{n}}\right)}} ,
\]
 where $\mathbb I_{n}$ is the $n\times n$ identity matrix. We have for the
defect $\partial P_{n}=-1$.

We define a \emph{preprojective Kronecker module} $P$ as being a
direct sum of indecomposable preprojective modules: $P=P_{a_{1}}\oplus P_{a_{2}}\oplus\dots\oplus P_{a_{l}}$,
where we use the convention that $a_{1}\le a_{2}\leq\dots\leq a_{l}$.

The \emph{preinjective indecomposable Kronecker module}s are also
determined up to isomorphism by their dimension vector. For
$n\in\mathbb{N}$ we will denote by $I_{n}$ the indecomposable
preinjective module of dimension $(n,n+1)$. So $I_{0}$ and $I_{1}$
are the injective indecomposable modules ($P_{0}=S_{2}$ being
simple). It is known that (up to isomorphism) $I_{n}=( k^{n},
k^{n+1};f,g)$, where choosing the canonical basis in $ k^{n+1}$ and
$ k^{n}$, the matrix of $f: k^{n+1}\to k^{n}$ (respectively of $g:
k^{n+1}\to k^{n}$) is $\begin{pmatrix}\mathbb I_{n} &
0\end{pmatrix}$ (respectively $\begin{pmatrix}0 &
\mathbb I_{n}\end{pmatrix}$). Thus in this case
\[
I_{n}:\xymatrix{ k^{n} &  k^{n+1}\ar@<1ex>[l]^{(0\, \mathbb
I_{n})}\ar@<-0.8ex>[l]_{(\mathbb I_{n}\,0)}} ,
\]
and we have for the defect $\partial I_{n}=1$.

We define a \emph{preinjective Kronecker module} $I$ as being a direct
sum of indecomposable preinjective modules: $I=I_{a_{1}}\oplus I_{a_{2}}\oplus\dots\oplus I_{a_{l}}$,
where we use the convention that $a_{1}\ge a_{2}\ge\dots\ge a_{l}$.

The \emph{regular indecomposable Kronecker module}s are those
indecomposable modules $M\in\Mod k K$ which are neither
preprojective nor preinjective. We describe here shortly only the
case when the base field is algebraically closed. If $ k=\bar{ k}$
is algebraically closed, then the regular indecomposables are
\[
R_{\mu}(n):\xymatrix{ k^{n} &
 k^{n}\ar@<1ex>[l]^{\mathbb I_{n}}\ar@<-0.8ex>[l]_{J_{\mu,n}}} \mbox{
for }k\in\bar{ k}\mbox{ and }R_{\infty}(n):\xymatrix{ k^{n} &
 k^{n}\ar@<1ex>[l]^{J_{0,n}}\ar@<-0.8ex>[l]_{\mathbb I_{n}}} ,
\]
 where $J_{\mu,n}$ is the $n\times n$ Jordan block with eigenvalue $\mu$.
The dimension of a regular indecomposable will be $\Dim
R_{p}(n)=(n,n)$ and we have for the defect $\partial R_{p}(n)=0$,
where $p\in\bar{ k}\cup\{\infty\}$.

A module $R\in\Mod k K$ will be called a \emph{regular Kronecker
module} if it is a direct sum of regular indecomposables. If
$\mu=(\mu_{1},\mu_{2},\dots,\mu_{m})$ is a partition, then we use
the notation $R_{p}(\mu)=R_{p}(\mu_{1})\oplus
R_{p}(\mu_{1})\oplus\dots\oplus R_{p}(\mu_{m})$.

The category $\Mod k K$ is a is a \emph{Krull-Schmidt category},
meaning that every module $M\in\Mod k K$ has a unique decomposition
\[
M=(P_{c_{1}}\oplus\dots\oplus P_{c_{n}})\oplus(\oplus_{p\in\mathcal{P}}R_{p}(\mu^{(p)}))\oplus(I_{d_{1}}\oplus\dots\oplus I_{d_{m}}),
\]
 where
\begin{itemize}
\item $(c_{1},\dots,c_{n})$ is a finite increasing sequence of non-negative
integers;
\item $\mu^{(p)}=(\mu_{1},\dots,\mu_{t})$ is a nonzero partition for finitely
many $p\in\bar{ k}\cup\{\infty\}$;
\item $(d_{1},\dots,d_{m})$ is a finite decreasing sequence of non negative
integers.
\end{itemize}
The integer sequences $(c_{1},\dots,c_{n})$ and
$(d_{1},\dots,d_{m})$ together with the partitions $\mu^{(p)}$
corresponding to every $p\in\bar{ k}\cup\{\infty\}$ are called the
\emph{Kronecker invariants} of the module $M$. Hence Kronecker
invariants determine a module $M\in\Mod k K$ up to isomorphism.

The following lemmas are well-known:
\begin{lem}
\label{lem:Ses-dim-defect}If there is a short exact sequence $0\to M^{\prime}\to M\to M^{\dprime}\to0$
of Kronecker modules, then $\Dim M=\Dim M^{\prime}+\Dim M^{\dprime}$
and $\partial M=\partial M^{\prime}+\partial M^{\dprime}$.
\end{lem}

\begin{lem}
\label{lem:ext-closed}Preinjectives (respectively preprojectives)
are extension closed. This means that in a short exact sequence of
the form $0\to Y\to M\to Y^{\prime}\to0$ with $Y,Y^{\prime}$ preinjectives
(preprojectives) $M$ must be also preinjective (preprojective).
\end{lem}
The category of Kronecker modules has been extensively studied
because the Kronecker algebra is a very important example of a tame
hereditary algebra. Moreover, the category has also a geometric
interpretation, since it is derived equivalent with the category
$\mbox{Coh}(\mathbb{P}^{1}( k))$ of coherent sheaves on the
projective line (see \cite{Beil}).

\section{Morphisms and short exact sequences\label{sec:Morphisms-and-short}}

In what follows we compile a few of our recent results on morphisms
and short exact sequences of Kronecker modules required for the
proofs in the next section. We emphasize that the theorems stated
here are valid independently of the underlying field $ k$ (as shown
in \cite{SzSz}).

We present now the numerical criteria for the existence of a
monomorphism $f:I{}^{\prime}\to I$ where $I,I{}^{\prime}$ are
preinjectives. The proof relies on homological algebra and the
knowledge on the category $\Mod kK$.
\begin{thm}
[\cite{SzSz}]\label{thm:Preinjective-embedding}Suppose $d_{1}\geq...\geq d_{n}>0$
and $c_{1}\geq...\geq c_{m}>0$ are integers. We have a monomorphism
\[
f:I_{d_{1}}\oplus...\oplus I_{d_{n}}\oplus dI_{0}\to I_{c_{1}}\oplus...\oplus I_{c_{m}}\oplus cI_{0}
\]
 if and only if $d\leq c$ and $d_{i}+...+d_{n}\leq\sum_{c_{j}\leq d_{i}}c_{j}$
for $i=\overline{1,n}$ (the empty sum being 0).\end{thm}
\begin{rem}
\label{rem:Preinjective-embedding-multiplicatively}Using the notation
$I{}^{\prime}=(a_{0}I_{0})\oplus...\oplus(a_{n}I_{n})\oplus...$,
$I=(b_{0}I_{0})\oplus...\oplus(b_{n}I_{n})\oplus...$ we have a monomorphism
$f:I{}^{\prime}\to I$ if and only if
\begin{eqnarray*}
a_{0} & \leq & b_{0}\\
a_{1} & \leq & b_{1}\\
a_{1}+2a_{2} & \leq & b_{1}+2b_{2}\\
 & \vdots\\
a_{1}+2a_{2}+...+na_{n} & \leq & b_{1}+2b_{2}+...+nb_{n}\\
 & \vdots
\end{eqnarray*}

So one can see that in the preinjective case ``a kind of'' weighted
dominance describes the numerical criteria for the embedding (it is
well-known that dominance ordering plays a crucial role in partition
combinatorics).

One should also note (using Lemma \ref{lem:Ses-dim-defect}) that
if $a_{m}=0$ for all $m>n$, then we have an exact sequence of the
form $0\to I^{\prime}\to I\to\beta I_{0}\to0$ (with $\beta$ arbitrary)
if and only if $b_{m}=0$ for all $m>n$ and
\begin{eqnarray*}
a_{0} & \leq & b_{0}\\
a_{1} & \leq & b_{1}\\
a_{1}+2a_{2} & \leq & b_{1}+2b_{2}\\
 & \vdots\\
a_{1}+2a_{2}+...+na_{n} & = & b_{1}+2b_{2}+...+nb_{n}.
\end{eqnarray*}

\end{rem}

\begin{rem}
Theorem \ref{thm:Preinjective-embedding} can be easily dualized for
preprojectives.
\end{rem}
We are going to state some results on short exact sequences in terms
of extension monoid products, so let us introduce this notion
shortly. For $d\in\mathbb{N}^{2}$ let $M_{d}=\{[M]\vert M\in\Mod k
K,\Dim M=d\}$ be the set of isomorphism classes of Kronecker modules
of dimension $d$. Following Reineke in \cite{Rei} for subsets
$\mathcal{A}\subset M_{d}$, $\mathcal{B}\subset M_{e}$ we define
\[
\mathcal{A}*\mathcal{B}=\{[Y]\in M_{d+e}\,|\,\exists\,0\to N\to Y\to M\to0\text{ exact for some }[M]\in\mathcal{A},[N]\in\mathcal{B}\}.
\]
 So the product $\mathcal{A}*\mathcal{B}$ is the set of isoclasses
of all extensions of modules $M$ with $[M]\in\mathcal{A}$ by modules
$N$ with $[N]\in\mathcal{B}$. This is in fact Reineke's extension
monoid product using isomorphism classes of modules instead of modules.
It is important to know (see \cite{Rei}) that the product above is
associative, i.e. for $\mathcal{A}\subset M_{d}$, $\mathcal{B}\subset M_{e}$,
$\mathcal{C}\subset M_{f}$, we have $(\mathcal{A}*\mathcal{B})*\mathcal{C}=\mathcal{A}*(\mathcal{B}*\mathcal{C})$.
Also $\{[0]\}*\mathcal{A}=\mathcal{A}*\{[0]\}=\mathcal{A}$. We will
call the operation ``$*$'' simply the \emph{extension monoid product}.

Using a set of rules describing the extension monoid products of Kronecker
modules in various cases we have proved in \cite{Szo3} the following
property for the extension monoid product of a preinjective and a
preprojective Kronecker module:
\begin{thm}
\label{thm:PII-III-ses-equivalence}Let $q>n>0$, $d_{1}\ge\dots\ge d_{q}\ge0$,
$c_{1}\ge\dots\ge c_{q-n}\ge0$ and $0\leq a_{1}\leq\dots\leq a_{n}$
be non-negative integers. Then $[I_{c_{1}}\oplus\dots\oplus I_{c_{q-n}}]\in\{[I_{d_{1}}\oplus\dots\oplus I_{d_{q}}]\}*\{[P_{a_{1}}\oplus\dots\oplus P_{a_{n}}]\}$
if and only if $[I_{d_{1}+a_{n}+1}\oplus\dots\oplus I_{d_{q}+a_{n}+1}]\in\{[I_{a_{n}-a_{1}}\oplus\dots\oplus I_{a_{n}-a_{n-1}}\oplus I_{0}]\}*\{[I_{c_{1}+a_{n}+1}\oplus\dots\oplus I_{c_{q-n}+a_{n}+1}]\}$,
or equivalently there is a short exact sequence
\[
0\to P_{a_{1}}\oplus\dots\oplus P_{a_{n}}\to I_{c_{1}}\oplus\dots\oplus I_{c_{q-n}}\to I_{d_{1}}\oplus\dots\oplus I_{d_{q}}\to0
\]
 if and only if there is a short exact sequence
\[
0\to I_{c_{1}+a_{n}+1}\oplus\dots\oplus I_{c_{q-n}+a_{n}+1}\to I_{d_{1}+a_{n}+1}\oplus\dots\oplus I_{d_{q}+a_{n}+1}\to I_{a_{n}-a_{1}}\oplus\dots\oplus I_{a_{n}-a_{n-1}}\oplus I_{0}\to0.
\]

\end{thm}
We will use in the proof of our main theorem the following corollary
of Theorem \ref{thm:PII-III-ses-equivalence}, obtained by applying
the theorem in the special case when kernel in the first short exact
sequence is of the form $P_{0}\oplus\dots\oplus P_{0}$:
\begin{cor}
\label{cor:P0II-III0-ses-equivalence}Let $q>\alpha>0$, $d_{1}\ge\dots\ge d_{q}\ge0$
and $c_{1}\ge\dots\ge c_{q-n}\ge0$ be non-negative integers. Then
$[I_{c_{1}}\oplus\dots\oplus I_{c_{q-\alpha}}]\in\{[I_{d_{1}}\oplus\dots\oplus I_{d_{q}}]\}*\{[\alpha P_{0}]\}$
if and only if $[I_{d_{1}+1}\oplus\dots\oplus I_{d_{q}+1}]\in\{[\alpha I_{0}]\}*\{[I_{c_{1}+1}\oplus\dots\oplus I_{c_{q-\alpha}+1}]\}$,
or equivalently there is a short exact sequence
\[
0\to\alpha P_{0}\to I_{c_{1}}\oplus\dots\oplus I_{c_{q-\alpha}}\to I_{d_{1}}\oplus\dots\oplus I_{d_{q}}\to0
\]
 if and only if there is a short exact sequence
\[
0\to I_{c_{1}+1}\oplus\dots\oplus I_{c_{q-\alpha}+1}\to I_{d_{1}+1}\oplus\dots\oplus I_{d_{q}+1}\to\alpha I_{0}\to0.
\]

\end{cor}
In what follows we will work out how to construct a monomorphism (or
an epimorphism) $f:I^{\prime}\to I$ between two preinjective
Kronecker modules, once we have determined using Theorem
\ref{thm:Preinjective-embedding} that $I^{\prime}$ embeds in $I$ or
using Theorem \ref{thm:PII-III-ses-equivalence} that $I^{\prime}$
projects on $I$. With $d_{1}\geq...\geq d_{n}>0$ and
$c_{1}\geq...\geq c_{m}>0$, let $I^{\prime}=I_{d_{1}}\oplus...\oplus
I_{d_{n}}\oplus dI_{0}$ and $I=I_{c_{1}}\oplus...\oplus
I_{c_{m}}\oplus cI_{0}$. We also use $p=\sum_{i=1}^{n}d_{i}$ and
$s=\sum_{i=1}^{m}c_{i}$, hence $\Dim I^{\prime}=(p,p+\partial
I^{\prime})=(p,p+n+d)$ and $\Dim I=(s,s+\partial I)=(s,s+m+c)$.

Identifying the modules with their representations it is clear that
we have a monomorphism (or an epimorphism) $f:I^{\prime}\to I$ if
and only if there exists a pair of full-rank matrices $(G,H)$ such
that the following diagram is commutative:
\[
\xymatrix{ k^{p}\ar@<-0.5ex>[d]_{G} &  k^{p+n+d}\ar@<0.5ex>[l]^{A^{\prime}}\ar@<-0.5ex>[l]_{A}\ar@<-3ex>[d]^{H}\\
 k^{s} &  k^{s+m+c}\ar@<0.5ex>[l]^{B^{\prime}}\ar@<-0.5ex>[l]_{B}
}
\]
On the diagram above we have the following matrices:
$A,A'\in\mathcal{M}_{p,p+n+d}( k)$, $B,B'\in\mathcal{M}_{s,s+m+c}(
k)$, $G\in\mathcal{M}_{s,p}( k)$, $H\in\mathcal{M}_{s+m+c,p+n+d}(
k)$, where \medskip{}
 \[ A = \left( \begin{array}{c:ccc}   & A_1 & \\   \Big{\underset{{\scriptscriptstyle (p\times d)}}{0}} & & \ddots & \\   \quad\; & & & A_n \\ \end{array} \right)\textrm{, } \; A' = \left( \begin{array}{c:ccc}   & A_1' & \\   \Big{\underset{{\scriptscriptstyle (p\times d)}}{0}} & & \ddots & \\   \quad\; & & & A_n' \\ \end{array} \right)\textrm{ with } \;  \begin{array}{r@{{}\mathrel{=}{}}l}     A_i & \begin{pmatrix}\mathbb I_{d_{i}} & 0\end{pmatrix} \\[\jot]     A_i' & \begin{pmatrix}0 & \mathbb I_{d_{i}}\end{pmatrix} \\   \end{array}\textrm{, } i=1,\dots ,n \]\medskip{}
\[ B = \left( \begin{array}{c:ccc}   & B_1 & \\   \Big{\underset{{\scriptscriptstyle (s\times c)}}{0}} & & \ddots & \\   \quad\; & & & B_m \\ \end{array} \right)\textrm{, } \; B' = \left( \begin{array}{c:ccc}   & B_1' & \\   \Big{\underset{{\scriptscriptstyle (s\times c)}}{0}} & & \ddots & \\   \quad\; & & & B_m' \\ \end{array} \right)\textrm{ with } \;  \begin{array}{r@{{}\mathrel{=}{}}l}     B_j & \begin{pmatrix}\mathbb I_{c_{j}} & 0\end{pmatrix} \\[\jot]     B_j' & \begin{pmatrix}0 & \mathbb I_{c_{j}}\end{pmatrix} \\   \end{array}\textrm{, } j=1,\dots ,m\]\medskip{}
\[ G=\left(\begin{array}{ccc}   G_{11} & \cdots & G_{1n}\\  \vdots & \ddots & \vdots\\  G_{m1} & \cdots & G_{mn}    \end{array}\right)\textrm{, }\;  \setlength{\dashlinegap}{2pt} H =  {\left( \begin{array}{c:ccc}  H_{00} & H_{01} & \cdots & H_{01} \\  \hdashline  H_{10} &  H_{11} & \cdots & H_{1n}\\   \vdots & \vdots & \ddots & \vdots \\  H_{m0}    & H_{m1} & \cdots & H_{mn}   \end{array} \right)}{} \]
\\
with the blocks $G_{ij}\in\mathcal{M}_{c_{i},d_{j}}( k)$,
$H_{00}\in\mathcal{M}_{c,d}( k)$, $H_{i0}\in\mathcal{M}_{c_{i}+1,d}(
k)$, $H_{0j}\in\mathcal{M}_{c,d_{j}+1}( k)$,
$H_{ij}\in\mathcal{M}_{c_{i}+1,d_{j}+1}( k)$ for $i=1,\dots,n$,
$j=1,\dots,m$. Commutativity of the diagram means the matrices $G$
and $H$ have to satisfy the following equalities: $BH=GA$ and
$B'H=GA'$. Writing out these equations using block-wise
multiplication we immediately get that $H_{j0}=0_{c,d_{j}+j}$, while
there are no restrictions in choosing $H_{00}\in\mathcal{M}_{c,d}(
k)$ and $H_{i0}\in\mathcal{M}_{c_{i}+1,d}( k)$.

For $i\in\{1\dots n\}$ and $j\in\{1\dots m\}$ let us write for the
corresponding blocks $G_{ij}=\left(g_{i',j'}\right)_{c_{i}\times d_{j}}$
and $H_{ij}=\left(h_{i',j'}\right)_{(c_{i}+1)\times(d_{j}+1)}$ and
expand the equations:
\[
\left\{ \begin{gathered}B_{i}H_{ij}=G_{ij}A_{j}\\
B_{i}^{\prime}H_{ij}=G_{ij}A'_{j}
\end{gathered}
\right.\iff\begin{cases}
h_{i',d_{j}+1}=0 & i'=1,\dots,c_{i}\\
h_{i',,j'}=g_{i',,j'} & i'=1,\dots,c_{i},j'=1,\dots,d_{j}\\
h_{i'+1,1}=0 & i'=1,\dots,c_{i}\\
h_{i'+1,j'+1}=g_{i',j'} & i'=1,\dots,c_{i},j'=1,\dots,d_{j}
\end{cases}.
\]
 So the entries in the blocks $H_{ij}$ and $G_{ij}$ must satisfy
the following relations: $h_{i',,j'}=h_{i'+1,j'+1}=g_{i',j'}$, $i'=1,\dots,c_{i}$,
$j'=1,\dots,d_{j}$ and consequently $g_{i',j'}=g_{i'+1,j'+1}$, $i'=1,\dots,c_{i}-1$,
$j'=1,\dots,d_{j}-1$ (i.e. the elements along all top-left to bottom-right
diagonals in the blocks $H_{ij}$ and $G_{ij}$ are equal). Using
$h_{i'+1,d_{j}+1}=g_{i',d_{j}}=0$ and $h_{i'+1,1}=g_{i'+1,1}=0$
for $i'=1,\dots,c_{i}-1$, we can explicitly give the elements of
the blocks $H_{ij}$ and $G_{ij}$ as:
\begin{equation}
h_{i',j'}=\begin{cases}
0 & j'-i'\notin\{0,\dots,d_{j}-c_{i}\}\\
\gamma_{j'-i'} & j'-i'\in\{0,\dots,d_{j}-c_{i}\}
\end{cases},\ i'=1,\dots,c_{i}+1,\ j'=1,\dots,d_{j}+1\label{eq:h}
\end{equation}
 and
\begin{equation}
g_{i',j'}=\begin{cases}
0 & j'-i'\notin\{0,\dots,d_{j}-c_{i}\}\\
\gamma_{j'-i'} & j'-i'\in\{0,\dots,d_{j}-c_{i}\}
\end{cases},\ i'=1,\dots,c_{i},\ j'=1,\dots,d_{j},\label{eq:g}
\end{equation}
 where $\gamma_{t}\in k$, $t\in\{0,\dots,d_{j}-c_{i}\}$. If
$d_{j}\ge c_{i}$ and at least one value $\gamma_{j'-i'}\neq0$, then
the blocks $H_{ij}$ and $G_{ij}$ both have full-rank. With this
information in mind (and the fact that the elements of the the top-left
block $H_{00}$ may be chosen arbitrarily) it is straightforward to
construct the full-rank matrices $H$ and $G$.

\section{Matrix pencils as Kronecker modules}

Next we will translate all the terms taken from pencil theory into
the language of Kronecker modules (representations). Indeed one can
easily see that a matrix pencil $A+\lambda B\in\mathcal{M}_{m,n}(
k[\lambda])$ corresponds to the Kronecker module $M_{A,B}=( k^{m},
k^{n};f_{A},f_{B})$, where choosing the canonical basis in $ k^{n}$
and $ k^{m}$ the matrix of $f_{A}: k^{n}\to k^{m}$ (respectively of
$f_{B}: k^{n}\to k^{m}$) is $A$ (respectively $B$). The strict
equivalence $A+\lambda B\sim A^{\prime}+\lambda B^{\prime}$ means
the isomorphism of modules $M_{A,B}\cong M_{A^{\prime},B^{\prime}}$.
It is also known that a pencil $A^{\prime}+\lambda B^{\prime}$ is a
subpencil of $A+\lambda B$ if and only if the module $M_{A',B'}$ is
a subfactor of $M_{A,B}$ (see \cite{Han}).

It is also clear that we have the following correspondence between
the classical Kronecker invariants and the Kronecker invariants (for
Kronecker modules) introduced in Section
\ref{sec:Kronecker-modules}: the minimal indices for rows correspond
to the integers $(c_{1},...,c_{n})$ parameterizing the preprojective
part, the minimal indices for columns correspond to the integers
$(d_{1},...,d_{m})$ parameterizing the preinjective part, the finite
elementary divisors correspond to the nonzero partitions
$\mu^{(p)}$, $p\in\bar{ k}$ and the infinite elementary divisors to
the partition $\mu^{(\infty)}$ (more precisely the partition
$\mu^{(p)}$ describes the dimensions of the Jordan blocks
corresponding to $p$).

Based on \cite{Han} we easily obtain the following :
\begin{prop}
\label{thm:Subpencil-iff-I0P0}$A^{\prime}+\lambda
B^{\prime}\in\mathcal{M}_{m^{\prime},n^{\prime}}( k[\lambda])$ is a
subpencil of $A+\lambda B\in\mathcal{M}_{m,n}( k[\lambda])$ if and
only if $m\ge m^{\prime}$, $n\ge n^{\prime}$ and
$[M_{A,B}]\in\{[(n-n^{\prime})I_{0}]\}*\{[M_{A^{\prime},B^{\prime}}]\}*\{[(m-m^{\prime})P_{0}]\}$.
\end{prop}
\begin{proof}
First note that
$[M_{A,B}]\in\{[(n-n^{\prime})I_{0}]\}*\{[M_{A^{\prime},B^{\prime}}]\}*\{[(m-m^{\prime})P_{0}]\}$
if and only if $\exists L\in\Mod k K$ such that
$[L]\in\{[(n-n^{\prime})I_{0}]\}*\{[M_{A^{\prime},B^{\prime}}]\}$
and $[M_{A,B}]\in\{[L]\}*\{[(m-m^{\prime})P_{0}]\}$. Also, the
condition $m\ge m^{\prime}$ and $n\ge n^{\prime}$ is clear.

``$\Longleftarrow$'' From $[L]\in\{[(n-n^{\prime})I_{0}]\}*\{[M_{A^{\prime},B^{\prime}}]\}$
and $[M_{A,B}]\in\{[L]\}*\{[(m-m^{\prime})P_{0}]\}$ we deduce the
existence of the short exact sequences
\[
0\to M_{A^{\prime},B^{\prime}}\to L\to(n-n^{\prime})I_{0}\to0
\]
 and
\[
0\to(m-m^{\prime})P_{0}\to M_{A,B}\to L\to0,
\]
 hence there exist an embedding and a projection $M_{A^{\prime},B^{\prime}}\hookrightarrow L\twoheadleftarrow M_{A,B}$,
i.e. $M_{A^{\prime},B^{\prime}}$ a subfactor of $M_{A,B}$ and the
fact that $A^{\prime}+\lambda B^{\prime}$ is a subpencil of $A+\lambda B$
follows.

``$\Longrightarrow$'' If $A^{\prime}+\lambda
B^{\prime}\in\mathcal{M}_{m^{\prime},n^{\prime}}( k[\lambda])$ is a
subpencil of $A+\lambda B\in\mathcal{M}_{m,n}( k[\lambda])$, then
there exist completion matrices $A_{12}+\lambda B_{12}$,
$A_{21}+\lambda B_{21}$, $A_{22}+\lambda B_{22}$ such that
\[
A+\lambda B\sim\left(\begin{array}{cc}
A^{\prime}+\lambda B^{\prime} & A_{12}+\lambda B_{12}\\
A_{21}+\lambda B_{21} & A_{22}+\lambda B_{22}
\end{array}\right)=\tilde{A}+\lambda\tilde{B},
\]
 that is, the modules $M_{A,B}$ and $M_{\tilde{A},\tilde{B}}$ are
isomorphic.

Consider the following short exact sequences, where we have identified
the module $M_{\tilde{A},\tilde{B}}$ with $M_{A,B}$:
\begin{center}
$ \xymatrix@=65pt@R=45pt{ 0\ar[d] & 0\ar[d]\\
\kappa\save[]+<-1.5cm,0.0cm>*{(m-m^\prime)P_0:} \restore
\save[]+<0.6cm,0.08cm>*{{}^{m-m^\prime}}\restore
\ar[d]_{\left(\begin{smallmatrix}0\\\mathbb I
\end{smallmatrix}\right)} & 0\ar[d]\\
\kappa\save[]+<-0.94cm,0.02cm>*{M_{A,B}:} \restore
\save[]+<0.3cm,0.03cm>*{{}^{m}}\restore
\ar[d]_{\left(\begin{smallmatrix}\mathbb I &
0\end{smallmatrix}\right)} &
\kappa\save[]+<0.26cm,0.03cm>*{{}^{n}}\restore\ar[d]^{\mathbb I}\\
\kappa\save[]+<-0.59cm,0.02cm>*{L:} \restore
\save[]+<0.35cm,0.08cm>*{{}^{m^\prime}}\restore \ar[d] &
\kappa\save[]+<0.26cm,0.03cm>*{{}^{n}}\restore\ar[d]\\  0 & 0  \POS
"2,2",\ar@{}"2,1",\ar@<-1ex>"2,1"!E-/25pt/*{} \POS
"2,2",\ar@{}"2,1",\ar@<+1ex>"2,1"!E-/25pt/*{} \POS
"3,2",\ar@{}"3,1",\ar@<-1ex>"3,1"!E-/9pt/*{}_{\left(\begin{smallmatrix}A^{\prime}
& A_{12}\\ A_{21} & A_{22} \end{smallmatrix}\right)} \POS
"3,2",\ar@{}"3,1",\ar@<+1ex>"3,1"!E-/9pt/*{}^{\left(\begin{smallmatrix}B^{\prime}
& B_{12}\\ B_{21} & B_{22} \end{smallmatrix}\right)} \POS
"4,2",\ar@{}"4,1",\ar@<-1ex>"4,1"!E-/9pt/*{}_{\left(\begin{smallmatrix}A^{\prime}
& A_{12}\end{smallmatrix}\right)} \POS
"4,2",\ar@{}"4,1",\ar@<+1ex>"4,1"!E-/9pt/*{}^{\left(\begin{smallmatrix}B^{\prime}
& B_{12}\end{smallmatrix}\right)} } $ \qquad{}\qquad{} $
\xymatrix@=65pt@R=45pt{ 0\ar[d] & 0\ar[d]\\
\kappa\save[]+<-1.07cm,0.02cm>*{M_{A^\prime,B^\prime}:} \restore
\save[]+<0.35cm,0.08cm>*{{}^{m^\prime}}\restore \ar[d]_{\mathbb I} &
\kappa\save[]+<0.31cm,0.08cm>*{{}^{n^\prime}}\restore\ar[d]^{\left(\begin{smallmatrix}\mathbb I\\
0 \end{smallmatrix}\right)}\\  \kappa\save[]+<-0.59cm,0.02cm>*{L:}
\restore \save[]+<0.35cm,0.08cm>*{{}^{m^\prime}}\restore \ar[d] &
\kappa\save[]+<0.26cm,0.03cm>*{{}^{n}}\restore\ar[d]^{\left(\begin{smallmatrix}0
& \mathbb I\end{smallmatrix}\right)}\\
0\save[]+<-1.36cm,0.0cm>*{(n-n^\prime)I_0:} \restore \ar[d] & \kappa
\save[]+<0.51cm,0.08cm>*{{}^{n-n^\prime}}\restore
\ar@<1ex>[l]\ar@<-1ex>[l]\ar[d]\\  0 & 0  \POS
"2,2",\ar@{}"2,1",\ar@<-1ex>"2,1"!E-/9pt/*{}_{A^{\prime}} \POS
"2,2",\ar@{}"2,1",\ar@<+1ex>"2,1"!E-/9pt/*{}^{B^{\prime}} \POS
"3,2",\ar@{}"3,1",\ar@<-1ex>"3,1"!E-/9pt/*{}_{\left(\begin{smallmatrix}A^{\prime}
& A_{12}\end{smallmatrix}\right)} \POS
"3,2",\ar@{}"3,1",\ar@<+1ex>"3,1"!E-/9pt/*{}^{\left(\begin{smallmatrix}B^{\prime}
& B_{12}\end{smallmatrix}\right)} } $
\par\end{center}

As it can be seen from these two short exact sequences, the module
$L$ may be constructed such that $[L]\in\{[(n-n^{\prime})I_{0}]\}*\{[M_{A^{\prime},B^{\prime}}]\}$
and $[M_{A,B}]\in\{[L]\}*\{[(m-m^{\prime})P_{0}]\}$, so $[M_{A,B}]\in\{[(n-n^{\prime})I_{0}]\}*\{[M_{A^{\prime},B^{\prime}}]\}*\{[(m-m^{\prime})P_{0}]\}$
follows immediately.
\end{proof}

\section{Complete solution to the subpencil problem in a particular case}

Let us consider matrix pencils $A+\lambda B$, $A^{\prime}+\lambda
B^{\prime}$ over $ k$, having only minimal indices for columns among
their classical Kronecker invariants. In this case, $A+\lambda
B\sim\diag(L_{\varepsilon_{1}},\dots,L_{\varepsilon_{p}})$ and
$A^{\prime}+\lambda
B^{\prime}\sim\diag(L_{\varepsilon_{1}^{\prime}},\dots,L_{\varepsilon_{q}^{\prime}})$,
where $0\leq\varepsilon_{1}\leq\dots\leq\varepsilon_{p}$ and
$0\leq\varepsilon_{1}^{\prime}\leq\dots\leq\varepsilon_{q}^{\prime}$
are the minimal indices for columns and
\[
L_{\varepsilon}:=\begin{pmatrix}\lambda & 1\\
 & \lambda & 1\\
 &  & \ddots & \ddots\\
 &  &  & \lambda & 1
\end{pmatrix}\in\mathcal{M}_{\varepsilon,\varepsilon+1}( k[\lambda])
\]
are the corresponding blocks on the diagonal (for further details
see \cite{Gant}). Hence, as explained before, one may identify the
pencil $A+\lambda B$ with the preinjective module
$M_{A,B}=I=I_{\varepsilon_{p}}\oplus\dots\oplus
I_{\varepsilon_{1}}\in\Mod k K$ and the pencil $A^{\prime}+\lambda
B^{\prime}$ with the preinjective module
$M_{A^{\prime},B^{\prime}}=I^{\prime}=I_{\varepsilon_{q}^{\prime}}\oplus\dots\oplus
I_{\varepsilon_{1}^{\prime}}\in\Mod k K$. Using this identification,
we have that $A^{\prime}+\lambda B^{\prime}$ is a subpencil of
$A+\lambda B$ if and only if $I^{\prime}$ is a subfactor of $I$,
that is if and only if there exists a Kronecker module $L\in\Mod k
K$ such that $I^{\prime}\hookrightarrow L\twoheadleftarrow I$. We
have the following theorem (where $\lfloor x\rfloor$ denotes the
integer part of $x$):
\begin{thm}
\label{thm:Preinjective-subpencil}If $I^{\prime}=a_{n}I_{n}\oplus\dots\oplus a_{0}I_{0}$
and $I=c_{n}I_{n}\oplus\dots\oplus c_{0}I_{0}$ are preinjective Kronecker
modules, then $I^{\prime}$ is a subfactor of $I$ (i.e. $\exists L$
such that $I^{\prime}\hookrightarrow L\twoheadleftarrow I$) if and
only if
\begin{eqnarray*}
b_{1}\leq\frac{1}{2}\left(\sum_{i=1}^{n}(i+1)c_{i}-\sum_{i=2}^{n}(i+1)b_{i}\right) & \mbox{and} & b_{0}\ge a_{0},
\end{eqnarray*}
 where the sequence $b_{n},\dots,b_{0}$ is defined by the following
(decreasing) recursion:
\[
b_{t}=\begin{cases}
\min(a_{n},c_{n})\vphantom{\biggl(\biggr)} & t=n\\
\left\lfloor \min\left(\frac{\sum_{i=t}^{n}ia_{i}-\sum_{i=t+1}^{n}ib_{i}}{t},\frac{\sum_{i=t}^{n}(i+1)c_{i}-\sum_{i=t+1}^{n}(i+1)b_{i}}{t+1}\right)\right\rfloor \vphantom{\vphantom{\biggl(\biggr)}} & 2\leq t<n\\
\sum_{i=1}^{n}ia_{i}-\sum_{i=2}^{n}ib_{i}\vphantom{\biggl(\biggr)} & t=1\\
\sum_{i=0}^{n}(i+1)c_{i}-\sum_{i=1}^{n}(i+1)b_{i}\vphantom{\biggl(\biggr)}
& t=0
\end{cases}.
\]
 Moreover, in this case the values $b_{n},\dots,b_{0}$ are non-negative
and one of the linking modules is $L=b_{n}I_{n}\oplus\dots\oplus
b_{0}I_{0}$. Note also that in pencil language $L$ is obtained from
$I'$ by column completion and $I$ is obtained from $L$ by row
completion.\end{thm}
\begin{proof}
First we show that $b_{n},\dots,b_{1}\ge0$. For
$b_{n}=\min(a_{n},c_{n})$, the inequality $b_{n}\ge0$ holds. Suppose
$b_{t}\ge0$ holds for all $t,$ where $l\leq t\leq n$,
$l\in\{2,\dots,n\}$. We are going to show that $b_{t-1}\ge0$ is also
true. Since
\[
b_{t}=\left\lfloor
\min\left(\frac{\sum_{i=t}^{n}ia_{i}-\sum_{i=t+1}^{n}ib_{i}}{t},\frac{\sum_{i=t}^{n}(i+1)c_{i}-\sum_{i=t+1}^{n}(i+1)b_{i}}{t+1}\right)\right\rfloor
,
\]
 follows that $0\leq tb_{t}\leq\sum_{i=t}^{n}ia_{i}-\sum_{i=t+1}^{n}ib_{i}$.
So
$\sum_{i=t-1}^{n}ia_{i}-\sum_{i=t}^{n}ib_{i}=(t-1)a_{t-1}+\sum_{i=t}^{n}ia_{i}-\sum_{i=t+1}^{n}ib_{i}-tb_{t}\ge(t-1)a_{t-1}\ge0$.
In the same way
$\sum_{i=t-1}^{n}(i+1)c_{i}-\sum_{i=t}^{n}(i+1)b_{i}=tc_{t-1}+\sum_{i=t}^{n}(i+1)c_{i}-\sum_{i=t+1}^{n}(i+1)b_{i}-(t+1)b_{t}\ge
tc_{t-1}\ge0$, so $b_{t-1}\ge0$.

From Proposition \ref{thm:Subpencil-iff-I0P0} we know that $I^{\prime}$
is a subfactor of $I$ if and only if $[I]\in\{[\beta
I_{0}]\}*\{[I^{\prime}]\}*\{[\alpha P_{0}]\}$ for some
$\alpha,\beta\in\mathbb{N}$. The extension monoid product is
associative, so $[I]\in\{[\beta I_{0}]\}*\{[I^{\prime}]\}*\{[\alpha
P_{0}]\}$ if and only if $[I]\in\{[L]\}*\{[\alpha P_{0}]\}$ for some
$L\in\Mod k K$, where $[L]\in\{[\beta I_{0}]\}*\{[I^{\prime}]\}$.
Since $[L]\in\{[\beta I_{0}]\}*\{[I^{\prime}]\}$ and $I^{\prime}$ is
preinjective, it results from Lemma \ref{lem:ext-closed} that
$L\in\Mod k K$ must also be preinjective.

Let us use now the multiplicative notation for $L$ as well, that
is, let use suppose $L=\dots\oplus u_{n}I_{n}\oplus\dots\oplus u_{0}I_{0}$
(we have no requirement for $u_{n}$ to be non-zero).

On one hand we have $[L]\in\{[\beta I_{0}]\}*\{[I^{\prime}]\}$ (with
$\beta$ arbitrary) if and only if we have the short exact sequence
$0\to I^{\prime}\to L\to\beta I_{0}\to0$, that is if and only if
$u_{m}=0$ for all $m>n$ and the following is true (see Remark \ref{rem:Preinjective-embedding-multiplicatively}):
\begin{eqnarray*}
a_{0} & \leq & u_{0}\\
a_{1} & \leq & u_{1}\\
a_{1}+2a_{2} & \leq & u_{1}+2u_{2}\\
 & \vdots\\
a_{1}+2a_{2}+...+(n-1)a_{n-1} & \leq & u_{1}+2u_{2}+...+(n-1)u_{n-1}\\
a_{1}+2a_{2}+...+(n-1)a_{n-1}+na_{n} & = & u_{1}+2u_{2}+...+(n-1)u_{n-1}+nu_{n}.
\end{eqnarray*}
 On the other hand, using Corollary \ref{cor:P0II-III0-ses-equivalence},
we have $[I]\in\{[L]\}*\{[\alpha P_{0}]\}$ if and only if $[L^{(+1)}]\in\{[\alpha I_{0}]\}*\{[I^{(+1)}]\}$,
if and only if we have the short exact sequence $0\to I^{(+1)}\to L^{(+1)}\to\alpha I_{0}\to0$,
where $L^{(+1)}=u_{n}I_{n+1}\oplus\dots\oplus u_{0}I_{1}$ and $I^{(+1)}=c_{n}I_{n+1}\oplus\dots\oplus c_{0}I_{1}$.
So this condition translates to
\begin{eqnarray*}
c_{0} & \leq & u_{0}\\
c_{0}+2c_{1} & \leq & u_{0}+2u_{1}\\
 & \vdots\\
c_{0}+2c_{1}+\dots+nc_{n-1} & \leq & u_{0}+2u_{1}+\dots+nu_{n-1}\\
c_{0}+2c_{1}+\dots+nc_{n-1}+(n+1)c_{n} & = & u_{0}+2u_{1}+\dots+nu_{n-1}+(n+1)u_{n}.
\end{eqnarray*}
 By extracting the inequalities from the last equality in both systems
and coupling them together we get that $[I]\in\{[\beta I_{0}]\}*\{[I^{\prime}]\}*\{[\alpha P_{0}]\}$
(with arbitrary $\alpha$ and $\beta$) if and only if there exist
non-negative integers $u_{0},u_{1},\dots,u_{n}$ such that the following
system is satisfied:
\begin{eqnarray*}
u_{0} & \ge & a_{0}\\
nu_{n} & \leq & na_{n}\\
(n-1)u_{n-1}+nu_{n} & \leq & (n-1)a_{n-1}+na_{n}\\
 & \vdots\\
2u_{2}+...+(n-1)u_{n-1}+nu_{n} & \leq & 2a_{2}+...+(n-1)a_{n-1}+na_{n}\\
u_{1}+2u_{2}+...+(n-1)u_{n-1}+nu_{n} & = & a_{1}+2a_{2}+...+(n-1)a_{n-1}+na_{n}\\
u_{0}+2u_{1}+\dots+nu_{n-1}+(n+1)u_{n} & = & c_{0}+2c_{1}+\dots+nc_{n-1}+(n+1)c_{n}\\
2u_{1}+\dots+nu_{n-1}+(n+1)u_{n} & \leq & 2c_{1}+\dots+nc_{n-1}+(n+1)c_{n}\\
 & \vdots\\
nu_{n-1}+(n+1)u_{n} & \leq & nc_{n-1}+(n+1)c_{n}\\
(n+1)u_{n} & \leq & (n+1)c_{n}.
\end{eqnarray*}
 Going further, we can write the system in the following equivalent
form:
\begin{eqnarray*}
u_{n} & \leq & \min(a_{n},c_{n})\\
u_{n-1} & \leq & \min\left(\frac{(n-1)a_{n-1}+na_{n}-nu_{n}}{n-1},\frac{nc_{n-1}+(n+1)c_{n}-(n+1)u_{n}}{n}\right)\\
 & \vdots\\
u_{2} & \leq & \min\left(\frac{\sum_{i=2}^{n}ia_{i}-\sum_{i=3}^{n}iu_{i}}{2},\frac{\sum_{i=2}^{n}(i+1)c_{i}-\sum_{i=3}^{n}(i+1)u_{i}}{3}\right)\\
u_{1} & = & \sum_{i=1}^{n}ia_{i}-\sum_{i=2}^{n}iu_{i}\leq\frac{1}{2}\left(\sum_{i=1}^{n}(i+1)c_{i}-\sum_{i=2}^{n}(i+1)u_{i}\right)\\
u_{0} & = & \sum_{i=0}^{n}(i+1)c_{i}-\sum_{i=1}^{n}(i+1)u_{i}\ge a_{0}
\end{eqnarray*}

``$\Longleftarrow$'' We have seen that the recursive definition
of the values $b_{n},\dots,b_{0}$ assure $b_{n},\dots,b_{1}\ge0$.
Moreover, if the inequalities
\begin{eqnarray*}
b_{1}\leq\frac{1}{2}\left(\sum_{i=1}^{n}(i+1)c_{i}-\sum_{i=2}^{n}(i+1)b_{i}\right) & \mbox{and} & b_{0}\ge a_{0}
\end{eqnarray*}
 are also satisfied, then we can take $(u_{0},u_{1},\dots,u_{n})=(b_{0},b_{1},\dots,b_{n})$,
which is a non-negative integer solution for the system.

``$\Longrightarrow$'' Suppose that there exists
$u_{0},u_{1},\dots,u_{n}\in\mathbb{N}$ such that the previous system
is satisfied (note that if $u_{2},\dots,u_{n}$ are chosen, then
$u_{0}$ and $u_{1}$ are determined) and consider the definition of
the sequence $b_{0},b_{1},\dots,b_{n}$ from the statement of the
theorem. If $(u_{2},u_{3},\dots,u_{n})\neq(b_{2},b_{3},\dots,b_{n})$
then let $t\in\{2,\dots,n\}$ be the greatest index such that
$u_{t}\neq b_{t}$. Then we must have
$u_{n}=b_{n},\dots,u_{t+1}=b_{t+1}$ and $u_{t}<b_{t}$. Let us denote
$d=b_{t}-u_{t}>0$. We perform the following change of variables:
$u_{0}^{\prime}=u_{0},\dots,u_{t-3}^{\prime}=u_{t-3}$,
$u_{t-2}^{\prime}=u_{t-2}+d$, $u_{t-1}^{\prime}=u_{t-1}-2d$,
$u_{t}^{\prime}=u_{t}+d$,
$u_{t+1}^{\prime}=u_{t+1},\dots,u_{n}^{\prime}=u_{n}$. Direct
calculations show that
$u_{0}^{\prime},u_{1}^{\prime},\dots,u_{n}^{\prime}\in\mathbb{Z}$
also satisfy the system, moreover
$u_{n}^{\prime}=b_{n},\dots,u_{t}^{\prime}=b_{t}$. Repeating this
process we find the following integer solution of the system:
$u_{n}^{\dprime}=b_{n},\dots,u_{2}^{\dprime}=b_{2},u_{1}^{\dprime},u_{0}^{\dprime}$.
The last two equations in the system guarantee that in fact
$u_{1}^{\dprime}=b_{1}$ and $u_{0}^{\dprime}=b_{0}$ hence they
satisfy the two inequalities from the statement of the theorem.

From the proof above one can see that a possible linking module is
$L=b_{n}I_{n}\oplus\dots\oplus b_{0}I_{0}$.\end{proof}
\begin{rem}
Theorem \ref{thm:Preinjective-subpencil} does not change if we take
preprojective modules instead of preinjectives.\end{rem}
\begin{example}
Consider the following matrix pencils over $\mathbb{C}$ written in
canonical diagonal form and having only minimal indices for columns
among their classical Kronecker invariants:
\[
A+\lambda B=\left(\begin{smallmatrix}\lambda & 1 & 0 & 0\\
0 & \lambda & 1 & 0\\
0 & 0 & \lambda & 1\\
 &  &  &  & \lambda & 1 & 0 & 0\\
 &  &  &  & 0 & \lambda & 1 & 0\\
 &  &  &  & 0 & 0 & \lambda & 1\\
 &  &  &  &  &  &  &  & \lambda & 1 & 0 & 0\\
 &  &  &  &  &  &  &  & 0 & \lambda & 1 & 0\\
 &  &  &  &  &  &  &  & 0 & 0 & \lambda & 1\\
 &  &  &  &  &  &  &  &  &  &  &  & \lambda & 1
\end{smallmatrix}\right)\in\mathcal{M}_{10,14}(\mathbb{C}[\lambda])
\]
 and
\[
A^{\prime}+\lambda B^{\prime}=\left(\begin{smallmatrix}0 & \lambda & 1 & 0 & 0 & 0 & 0\\
 & 0 & \lambda & 1 & 0 & 0 & 0\\
 & 0 & 0 & \lambda & 1 & 0 & 0\\
 & 0 & 0 & 0 & \lambda & 1 & 0\\
 & 0 & 0 & 0 & 0 & \lambda & 1\\
 &  &  &  &  &  &  & \lambda & 1 & 0\\
 &  &  &  &  &  &  & 0 & \lambda & 1\\
 &  &  &  &  &  &  &  &  &  & \lambda & 1
\end{smallmatrix}\right)\in\mathcal{M}_{8,12}(\mathbb{C}[\lambda]).
\]
 The pencil $A+\lambda B$ has $\varepsilon_{1}=\varepsilon_{2}=\varepsilon_{3}=3$,
and $\varepsilon_{4}=1$ as its minimal indices for columns, while
in the case of the pencil $A^{\prime}+\lambda B^{\prime}$ these are
$\varepsilon_{1}^{\prime}=0$, $\varepsilon_{2}^{\prime}=5$, $\varepsilon_{3}^{\prime}=2$
and $\varepsilon_{4}^{\prime}=1$. Hence the corresponding modules
are $M_{A,B}=I_{3}\oplus I_{3}\oplus I_{3}\oplus I_{1}$ and $M_{A^{\prime},B^{\prime}}=I_{5}\oplus I_{2}\oplus I_{1}\oplus I_{0}$.
Written using the multiplicative notation used in Theorem \ref{thm:Preinjective-subpencil},
$M_{A^{\prime},B^{\prime}}=\bigoplus_{i=0}^{5}a_{i}I_{i}$ and $M_{A,B}=\bigoplus_{i=0}^{5}c_{i}I_{i}$,
where $(a_{0},a_{1},\dots,a_{5})=(1,1,1,0,0,1)$ and $(c_{0},c_{1},\dots,c_{5})=(0,1,0,3,0,0)$.
We use the recursive formula from the theorem to compute the sequence
$(b_{0},b_{1},\dots,b_{5})=(2,1,2,1,0,0)$ and to find out that the
inequalities
\begin{eqnarray*}
b_{1}\leq\frac{1}{2}\left(\sum_{i=1}^{5}(i+1)c_{i}-\sum_{i=2}^{5}(i+1)b_{i}\right) & \mbox{and} & b_{0}\ge a_{0}
\end{eqnarray*}
 are satisfied. So $A^{\prime}+\lambda B^{\prime}$ is a subpencil
of $A+\lambda B$ or equivalently, $M_{A^{\prime},B^{\prime}}$ is
a subfactor of $M_{A,B}$, i.e. $\exists L$ such that $M_{A^{\prime},B^{\prime}}\hookrightarrow L\twoheadleftarrow M_{A,B}$.
Moreover, we can take the linking module $L$ to be $L=\bigoplus_{i=0}^{5}b_{i}I_{i}=I_{3}\oplus I_{2}\oplus I_{2}\oplus I_{1}\oplus I_{0}\oplus I_{0}$.
We could use at this point Theorem \ref{thm:Preinjective-embedding}
to verify the existence of the embedding $M_{A^{\prime},B^{\prime}}\hookrightarrow L$
and Corollary \ref{cor:P0II-III0-ses-equivalence} to verify the existence
of the projection $L\twoheadleftarrow M_{A,B}$ with the kernel equal
to $2P_{0}$. The matrix pencil corresponding to the module $L$ is
\[
L_{1}+\lambda L_{2}=\left(\begin{smallmatrix}0 & 0 & \lambda & 1 & 0 & 0\\
 &  & 0 & \lambda & 1 & 0\\
 &  & 0 & 0 & \lambda & 1\\
 &  &  &  &  &  & \lambda & 1 & 0\\
 &  &  &  &  &  & 0 & \lambda & 1\\
 &  &  &  &  &  &  &  &  & \lambda & 1 & 0\\
 &  &  &  &  &  &  &  &  & 0 & \lambda & 1\\
 &  &  &  &  &  &  &  &  &  &  &  & \lambda & 1
\end{smallmatrix}\right)\in\mathcal{M}_{8,14}(\mathbb{C}[\lambda]).
\]

Let us construct now the completion matrices $A_{12}+\lambda B_{12}$,
$A_{21}+\lambda B_{21}$, $A_{22}+\lambda B_{22}$, i.e. those matrix
blocks for which the following equivalence holds:
\[
A+\lambda B\sim\left(\begin{array}{cc}
A^{\prime}+\lambda B^{\prime} & A_{12}+\lambda B_{12}\\
A_{21}+\lambda B_{21} & A_{22}+\lambda B_{22}
\end{array}\right).
\]

Since we have an embedding $M_{A^{\prime},B^{\prime}}\overset{f}{\hookrightarrow}L$,
we must have $f=(F_{1},F_{2})$, where $F_{1}\in\mathcal{M}_{14,12}(\mathbb{C})$
and $F_{2}\in\mathcal{M}_{8}(\mathbb{C})$ are full-rank matrices
such that $(L_{1}+\lambda L_{2})F_{1}=F_{2}(A^{\prime}+\lambda B^{\prime})$.
Also, for the projection $M_{A,B}\overset{g}{\twoheadrightarrow}L$,
we have $g=(G_{1},G_{2})$, where $G_{1}\in\mathcal{M}_{14}(\mathbb{C})$
and $G_{2}\in\mathcal{M}_{10,8}(\mathbb{C})$ are full-rank matrices
such that $(L_{1}+\lambda L_{2})G_{1}=G_{2}(A+\lambda B)$. Using
the method detailed at the end of Section \ref{sec:Morphisms-and-short}
we can construct these matrices as: \[ \setlength{\dashlinegap}{2pt}  F_{1}= \left( \begin{array}{c:cccccc:ccc:cc}
0 & 0 & 0 & 0 & 1 & 0 & 0 & 0 & 0 & 0 & 0 & 0\\

1 & 0 & 0 & 0 & 0 & 0 & 0 & 0 & 0 & 0 & 0 & 0\\ \hdashline

{\color{lightgray}0} & 1 & 0 & 0 & {\color{lightgray}0} & {\color{lightgray}0} & {\color{lightgray}0} & {\color{lightgray}0} & {\color{lightgray}0} & {\color{lightgray}0} & {\color{lightgray}0} & {\color{lightgray}0}\\

{\color{lightgray}0} & {\color{lightgray}0} & 1 & 0 & 0 & {\color{lightgray}0} & {\color{lightgray}0} & {\color{lightgray}0} & {\color{lightgray}0} & {\color{lightgray}0} & {\color{lightgray}0} & {\color{lightgray}0}\\

{\color{lightgray}0} & {\color{lightgray}0} & {\color{lightgray}0} & 1 & 0 & 0 & {\color{lightgray}0} & {\color{lightgray}0} & {\color{lightgray}0} & {\color{lightgray}0} & {\color{lightgray}0} & {\color{lightgray}0}\\

{\color{lightgray}0} & {\color{lightgray}0} & {\color{lightgray}0} & {\color{lightgray}0} & 1 & 0 & 0 & {\color{lightgray}0} & {\color{lightgray}0} & {\color{lightgray}0} & {\color{lightgray}0} & {\color{lightgray}0}\\ \hdashline

{\color{lightgray}0} & 0 & 0 & 0 & 1 & {\color{lightgray}0} & {\color{lightgray}0} & 0 & {\color{lightgray}0} & {\color{lightgray}0} & {\color{lightgray}0} & {\color{lightgray}0}\\

{\color{lightgray}0} & {\color{lightgray}0} & 0 & 0 & 0 & 1 & {\color{lightgray}0} & {\color{lightgray}0} & 0 & {\color{lightgray}0} & {\color{lightgray}0} & {\color{lightgray}0}\\

{\color{lightgray}0} & {\color{lightgray}0} & {\color{lightgray}0} & 0 & 0 & 0 & 1 & {\color{lightgray}0} & {\color{lightgray}0} & 0 & {\color{lightgray}0} & {\color{lightgray}0}\\ \hdashline

{\color{lightgray}0} & 0 & 0 & 0 & 0 & {\color{lightgray}0} & {\color{lightgray}0} & 1 & {\color{lightgray}0} & {\color{lightgray}0} & {\color{lightgray}0} & {\color{lightgray}0}\\

{\color{lightgray}0} & {\color{lightgray}0} & 0 & 0 & 0 & 0 & {\color{lightgray}0} & {\color{lightgray}0} & 1 & {\color{lightgray}0} & {\color{lightgray}0} & {\color{lightgray}0}\\

{\color{lightgray}0} & {\color{lightgray}0} & {\color{lightgray}0} & 0 & 0 & 0 & 0 & {\color{lightgray}0} & {\color{lightgray}0} & 1 & {\color{lightgray}0} & {\color{lightgray}0}\\ \hdashline

{\color{lightgray}0} & 0 & 0 & 0 & 0 & 0 & {\color{lightgray}0} & 0 & 0 & {\color{lightgray}0} & 1 & {\color{lightgray}0}\\

{\color{lightgray}0} & {\color{lightgray}0} & 0 & 0 & 0 & 0 & 0 & {\color{lightgray}0} & 0 & 0 & {\color{lightgray}0} & 1

\end{array}  \right)  \quad\mbox{and}\quad F_{2}=\left( \begin{array}{ccccc:cc:c}

1 & 0 & 0 & {\color{lightgray}0} & {\color{lightgray}0} & {\color{lightgray}0} & {\color{lightgray}0} & {\color{lightgray}0} \\
{\color{lightgray}0} & 1 & 0 & 0 & {\color{lightgray}0} & {\color{lightgray}0} & {\color{lightgray}0} & {\color{lightgray}0} \\
{\color{lightgray}0} & {\color{lightgray}0} & 1 & 0 & 0 & {\color{lightgray}0} & {\color{lightgray}0} & {\color{lightgray}0} \\ \hdashline

0 & 0 & 0 & 1 & {\color{lightgray}0} & 0 & {\color{lightgray}0} & {\color{lightgray}0} \\
{\color{lightgray}0} & 0 & 0 & 0 & 1 & {\color{lightgray}0} & 0 & {\color{lightgray}0} \\ \hdashline

0 & 0 & 0 & 0 & {\color{lightgray}0} & 1 & {\color{lightgray}0} & {\color{lightgray}0} \\
{\color{lightgray}0} & 0 & 0 & 0 & 0 & {\color{lightgray}0} & 1 & {\color{lightgray}0} \\ \hdashline

0 & 0 & 0 & 0 & 0 & 0 & 0 & 1

\end{array}\right) = \mathbb I_{8}, \] where $\mathbb I_{8}$ is the $8\times8$ identity matrix and \[ \setlength{\dashlinegap}{2pt} G_{1}= \left( \begin{array}{cccc:cccc:cccc:cc}

0 & 0 & 0 & 0 & 0 & 0 & 0 & 1 & 0 & 0 & 0 & 0 & 0 & 0\\

0 & 0 & 0 & 0 & 0 & 0 & 0 & 0 & 1 & 0 & 0 & 0 & 0 & 0\\ \hdashline

1 & {\color{lightgray}0} & {\color{lightgray}0} & {\color{lightgray}0} & 0 & {\color{lightgray}0} & {\color{lightgray}0} & {\color{lightgray}0} & 0 & {\color{lightgray}0} & {\color{lightgray}0} & {\color{lightgray}0} & {\color{lightgray}0} & {\color{lightgray}0}\\

{\color{lightgray}0} & 1 & {\color{lightgray}0} & {\color{lightgray}0} & {\color{lightgray}0} & 0 & {\color{lightgray}0} & {\color{lightgray}0} & {\color{lightgray}0} & 0 & {\color{lightgray}0} & {\color{lightgray}0} & {\color{lightgray}0} & {\color{lightgray}0}\\

{\color{lightgray}0} & {\color{lightgray}0} & 1 & {\color{lightgray}0} & {\color{lightgray}0} & {\color{lightgray}0} & 0 & {\color{lightgray}0} & {\color{lightgray}0} & {\color{lightgray}0} & 0 & {\color{lightgray}0} & {\color{lightgray}0} & {\color{lightgray}0}\\

{\color{lightgray}0} & {\color{lightgray}0} & {\color{lightgray}0} & 1 & {\color{lightgray}0} & {\color{lightgray}0} & {\color{lightgray}0} & 0 & {\color{lightgray}0} & {\color{lightgray}0} & {\color{lightgray}0} & 0 & {\color{lightgray}0} & {\color{lightgray}0}\\ \hdashline

0 & 0 & {\color{lightgray}0} & {\color{lightgray}0} & 1 & 0 & {\color{lightgray}0} & {\color{lightgray}0} & 0 & 0 & {\color{lightgray}0} & {\color{lightgray}0} & {\color{lightgray}0} & {\color{lightgray}0}\\

{\color{lightgray}0} & 0 & 0 & {\color{lightgray}0} & {\color{lightgray}0} & 1 & 0 & {\color{lightgray}0} & {\color{lightgray}0} & 0 & 0 & {\color{lightgray}0} & {\color{lightgray}0} & {\color{lightgray}0}\\

{\color{lightgray}0} & {\color{lightgray}0} & 0 & 0 & {\color{lightgray}0} & {\color{lightgray}0} & 1 & 0 & {\color{lightgray}0} & {\color{lightgray}0} & 0 & 0 & {\color{lightgray}0} & {\color{lightgray}0}\\ \hdashline

0 & 0 & {\color{lightgray}0} & {\color{lightgray}0} & 0 & 0 & {\color{lightgray}0} & {\color{lightgray}0} & 0 & 1 & {\color{lightgray}0} & {\color{lightgray}0} & {\color{lightgray}0} & {\color{lightgray}0}\\

{\color{lightgray}0} & 0 & 0 & {\color{lightgray}0} & {\color{lightgray}0} & 0 & 0 & {\color{lightgray}0} & {\color{lightgray}0} & 0 & 1 & {\color{lightgray}0} & {\color{lightgray}0} & {\color{lightgray}0}\\

{\color{lightgray}0} & {\color{lightgray}0} & 0 & 0 & {\color{lightgray}0} & {\color{lightgray}0} & 0 & 0 & {\color{lightgray}0} & {\color{lightgray}0} & 0 & 1 & {\color{lightgray}0} & {\color{lightgray}0}\\ \hdashline

0 & 0 & 0 & {\color{lightgray}0} & 0 & 0 & 0 & {\color{lightgray}0} & 0 & 0 & 0 & {\color{lightgray}0} & 1 & {\color{lightgray}0}\\

{\color{lightgray}0} & 0 & 0 & 0 & {\color{lightgray}0} & 0 & 0 & 0 & {\color{lightgray}0} & 0 & 0 & 0 & {\color{lightgray}0} & 1

\end{array} \right)\textrm{, } G_{2}=\left(\begin{array}{ccc:ccc:ccc:c}

1 & {\color{lightgray}0} & {\color{lightgray}0} & 0 & {\color{lightgray}0} & {\color{lightgray}0} & 0 & {\color{lightgray}0} & {\color{lightgray}0} & {\color{lightgray}0}\\

{\color{lightgray}0} & 1 & {\color{lightgray}0} & {\color{lightgray}0} & 0 & {\color{lightgray}0} & {\color{lightgray}0} & 0 & {\color{lightgray}0} & {\color{lightgray}0}\\

{\color{lightgray}0} & {\color{lightgray}0} & 1 & {\color{lightgray}0} & {\color{lightgray}0} & 0 & {\color{lightgray}0} & {\color{lightgray}0} & 0 & {\color{lightgray}0}\\ \hdashline

0 & 0 & {\color{lightgray}0} & 1 & 0 & {\color{lightgray}0} & 0 & 0 & {\color{lightgray}0} & {\color{lightgray}0}\\

{\color{lightgray}0} & 0 & 0 & {\color{lightgray}0} & 1 & 0 & {\color{lightgray}0} & 0 & 0 & {\color{lightgray}0}\\ \hdashline

0 & 0 & {\color{lightgray}0} & 0 & 0 & {\color{lightgray}0} & 0 & 1 & {\color{lightgray}0} & {\color{lightgray}0}\\

{\color{lightgray}0} & 0 & 0 & {\color{lightgray}0} & 0 & 0 & {\color{lightgray}0} & 0 & 1 & {\color{lightgray}0}\\ \hdashline

0 & 0 & 0 & 0 & 0 & 0 & 0 & 0 & 0 & 1

\end{array}\right). \] We have chosen the elements in the blocks according to the equations
\ref{eq:h} and \ref{eq:g}, respectively. In the matrices above we
have marked by gray elements which must be zero and used black where
we had a choice for the elements.

From now on we work along the proof of Proposition 1. in \cite{Han}.

Since $F_{2}$ and $G_{1}$ are full-rank square matrices, they are
invertible. In our case $F_{2}^{-1}=F_{2}$ and $G_{1}^{-1}=G_{1}^{\intercal}$.
The matrices $G_{1}^{-1}F_{1}$ and $F_{2}^{-1}G_{2}$ are also full-rank
matrices, so there are non-singular square matrices $C_{1}$, $C_{2}$,
$D_{1}$ and $D_{2}$ such that $G_{1}^{-1}F_{1}=C_{1}\begin{pmatrix}\mathbb I_{12}\\
0
\end{pmatrix}C_{2}$ and $F_{2}^{-1}G_{2}=D_{1}\begin{pmatrix}\mathbb I_{8} & 0\end{pmatrix}D_{2}$,
respectively. In our case these matrices are
\[
C_{1}=\left(\begin{smallmatrix}{\color{lightgray}0} & {\color{lightgray}0} & -1 & {\color{lightgray}0} & {\color{lightgray}0} & {\color{lightgray}0} & {\color{lightgray}0} & {\color{lightgray}0} & {\color{lightgray}0} & {\color{lightgray}0} & {\color{lightgray}0} & {\color{lightgray}0} & {\color{lightgray}0} & {\color{lightgray}0}\\
{\color{lightgray}0} & {\color{lightgray}0} & {\color{lightgray}0} & -1 & {\color{lightgray}0} & {\color{lightgray}0} & {\color{lightgray}0} & {\color{lightgray}0} & {\color{lightgray}0} & {\color{lightgray}0} & {\color{lightgray}0} & {\color{lightgray}0} & {\color{lightgray}0} & {\color{lightgray}0}\\
{\color{lightgray}0} & {\color{lightgray}0} & {\color{lightgray}0} & {\color{lightgray}0} & -1 & {\color{lightgray}0} & {\color{lightgray}0} & {\color{lightgray}0} & {\color{lightgray}0} & {\color{lightgray}0} & {\color{lightgray}0} & {\color{lightgray}0} & {\color{lightgray}0} & {\color{lightgray}0}\\
-1 & {\color{lightgray}0} & {\color{lightgray}0} & {\color{lightgray}0} & {\color{lightgray}0} & {\color{lightgray}0} & {\color{lightgray}0} & {\color{lightgray}0} & {\color{lightgray}0} & {\color{lightgray}0} & {\color{lightgray}0} & {\color{lightgray}0} & 1 & {\color{lightgray}0}\\
-1 & {\color{lightgray}0} & {\color{lightgray}0} & {\color{lightgray}0} & {\color{lightgray}0} & {\color{lightgray}0} & {\color{lightgray}0} & {\color{lightgray}0} & {\color{lightgray}0} & {\color{lightgray}0} & {\color{lightgray}0} & {\color{lightgray}0} & {\color{lightgray}0} & 1\\
{\color{lightgray}0} & {\color{lightgray}0} & {\color{lightgray}0} & {\color{lightgray}0} & {\color{lightgray}0} & 1 & {\color{lightgray}0} & {\color{lightgray}0} & {\color{lightgray}0} & {\color{lightgray}0} & {\color{lightgray}0} & {\color{lightgray}0} & {\color{lightgray}0} & {\color{lightgray}0}\\
{\color{lightgray}0} & {\color{lightgray}0} & {\color{lightgray}0} & {\color{lightgray}0} & {\color{lightgray}0} & {\color{lightgray}0} & 1 & {\color{lightgray}0} & {\color{lightgray}0} & {\color{lightgray}0} & {\color{lightgray}0} & {\color{lightgray}0} & {\color{lightgray}0} & {\color{lightgray}0}\\
-1 & {\color{lightgray}0} & {\color{lightgray}0} & {\color{lightgray}0} & {\color{lightgray}0} & {\color{lightgray}0} & {\color{lightgray}0} & {\color{lightgray}0} & {\color{lightgray}0} & {\color{lightgray}0} & {\color{lightgray}0} & {\color{lightgray}0} & {\color{lightgray}0} & {\color{lightgray}0}\\
{\color{lightgray}0} & -1 & {\color{lightgray}0} & {\color{lightgray}0} & {\color{lightgray}0} & {\color{lightgray}0} & {\color{lightgray}0} & {\color{lightgray}0} & {\color{lightgray}0} & {\color{lightgray}0} & {\color{lightgray}0} & {\color{lightgray}0} & {\color{lightgray}0} & {\color{lightgray}0}\\
{\color{lightgray}0} & {\color{lightgray}0} & {\color{lightgray}0} & {\color{lightgray}0} & {\color{lightgray}0} & {\color{lightgray}0} & {\color{lightgray}0} & -1 & {\color{lightgray}0} & {\color{lightgray}0} & {\color{lightgray}0} & {\color{lightgray}0} & {\color{lightgray}0} & {\color{lightgray}0}\\
{\color{lightgray}0} & {\color{lightgray}0} & {\color{lightgray}0} & {\color{lightgray}0} & {\color{lightgray}0} & {\color{lightgray}0} & {\color{lightgray}0} & {\color{lightgray}0} & -1 & {\color{lightgray}0} & {\color{lightgray}0} & {\color{lightgray}0} & {\color{lightgray}0} & {\color{lightgray}0}\\
{\color{lightgray}0} & {\color{lightgray}0} & {\color{lightgray}0} & {\color{lightgray}0} & {\color{lightgray}0} & {\color{lightgray}0} & {\color{lightgray}0} & {\color{lightgray}0} & {\color{lightgray}0} & -1 & {\color{lightgray}0} & {\color{lightgray}0} & {\color{lightgray}0} & {\color{lightgray}0}\\
{\color{lightgray}0} & {\color{lightgray}0} & {\color{lightgray}0} & {\color{lightgray}0} & {\color{lightgray}0} & {\color{lightgray}0} & {\color{lightgray}0} & {\color{lightgray}0} & {\color{lightgray}0} & {\color{lightgray}0} & -1 & {\color{lightgray}0} & {\color{lightgray}0} & {\color{lightgray}0}\\
{\color{lightgray}0} & {\color{lightgray}0} & {\color{lightgray}0} & {\color{lightgray}0} & {\color{lightgray}0} & {\color{lightgray}0} & {\color{lightgray}0} & {\color{lightgray}0} & {\color{lightgray}0} & {\color{lightgray}0} & {\color{lightgray}0} & -1 & {\color{lightgray}0} & {\color{lightgray}0}
\end{smallmatrix}\right),\; C_{2}=\left(\begin{smallmatrix}{\color{lightgray}0} & {\color{lightgray}0} & {\color{lightgray}0} & {\color{lightgray}0} & -1 & {\color{lightgray}0} & {\color{lightgray}0} & {\color{lightgray}0} & {\color{lightgray}0} & {\color{lightgray}0} & {\color{lightgray}0} & {\color{lightgray}0}\\
-1 & {\color{lightgray}0} & {\color{lightgray}0} & {\color{lightgray}0} & {\color{lightgray}0} & {\color{lightgray}0} & {\color{lightgray}0} & {\color{lightgray}0} & {\color{lightgray}0} & {\color{lightgray}0} & {\color{lightgray}0} & {\color{lightgray}0}\\
{\color{lightgray}0} & -1 & {\color{lightgray}0} & {\color{lightgray}0} & {\color{lightgray}0} & {\color{lightgray}0} & {\color{lightgray}0} & {\color{lightgray}0} & {\color{lightgray}0} & {\color{lightgray}0} & {\color{lightgray}0} & {\color{lightgray}0}\\
{\color{lightgray}0} & {\color{lightgray}0} & -1 & {\color{lightgray}0} & {\color{lightgray}0} & {\color{lightgray}0} & {\color{lightgray}0} & {\color{lightgray}0} & {\color{lightgray}0} & {\color{lightgray}0} & {\color{lightgray}0} & {\color{lightgray}0}\\
{\color{lightgray}0} & {\color{lightgray}0} & {\color{lightgray}0} & -1 & {\color{lightgray}0} & {\color{lightgray}0} & {\color{lightgray}0} & {\color{lightgray}0} & {\color{lightgray}0} & {\color{lightgray}0} & {\color{lightgray}0} & {\color{lightgray}0}\\
{\color{lightgray}0} & {\color{lightgray}0} & {\color{lightgray}0} & {\color{lightgray}0} & {\color{lightgray}0} & 1 & {\color{lightgray}0} & {\color{lightgray}0} & {\color{lightgray}0} & {\color{lightgray}0} & {\color{lightgray}0} & {\color{lightgray}0}\\
{\color{lightgray}0} & {\color{lightgray}0} & {\color{lightgray}0} & {\color{lightgray}0} & {\color{lightgray}0} & {\color{lightgray}0} & 1 & {\color{lightgray}0} & {\color{lightgray}0} & {\color{lightgray}0} & {\color{lightgray}0} & {\color{lightgray}0}\\
{\color{lightgray}0} & {\color{lightgray}0} & {\color{lightgray}0} & {\color{lightgray}0} & {\color{lightgray}0} & {\color{lightgray}0} & {\color{lightgray}0} & -1 & {\color{lightgray}0} & {\color{lightgray}0} & {\color{lightgray}0} & {\color{lightgray}0}\\
{\color{lightgray}0} & {\color{lightgray}0} & {\color{lightgray}0} & {\color{lightgray}0} & {\color{lightgray}0} & {\color{lightgray}0} & {\color{lightgray}0} & {\color{lightgray}0} & -1 & {\color{lightgray}0} & {\color{lightgray}0} & {\color{lightgray}0}\\
{\color{lightgray}0} & {\color{lightgray}0} & {\color{lightgray}0} & {\color{lightgray}0} & {\color{lightgray}0} & {\color{lightgray}0} & {\color{lightgray}0} & {\color{lightgray}0} & {\color{lightgray}0} & -1 & {\color{lightgray}0} & {\color{lightgray}0}\\
{\color{lightgray}0} & {\color{lightgray}0} & {\color{lightgray}0} & {\color{lightgray}0} & {\color{lightgray}0} & {\color{lightgray}0} & {\color{lightgray}0} & {\color{lightgray}0} & {\color{lightgray}0} & {\color{lightgray}0} & -1 & {\color{lightgray}0}\\
{\color{lightgray}0} & {\color{lightgray}0} & {\color{lightgray}0} & {\color{lightgray}0} & {\color{lightgray}0} & {\color{lightgray}0} & {\color{lightgray}0} & {\color{lightgray}0} & {\color{lightgray}0} & {\color{lightgray}0} & {\color{lightgray}0} & -1
\end{smallmatrix}\right),
\]

\[
D_{1}=\mathbb I_{8},\qquad\mbox{and}\qquad D_{2}=\left(\begin{smallmatrix}1 & {\color{lightgray}0} & {\color{lightgray}0} & {\color{lightgray}0} & {\color{lightgray}0} & {\color{lightgray}0} & {\color{lightgray}0} & {\color{lightgray}0} & {\color{lightgray}0} & {\color{lightgray}0}\\
{\color{lightgray}0} & 1 & {\color{lightgray}0} & {\color{lightgray}0} & {\color{lightgray}0} & {\color{lightgray}0} & {\color{lightgray}0} & {\color{lightgray}0} & {\color{lightgray}0} & {\color{lightgray}0}\\
{\color{lightgray}0} & {\color{lightgray}0} & 1 & {\color{lightgray}0} & {\color{lightgray}0} & {\color{lightgray}0} & {\color{lightgray}0} & {\color{lightgray}0} & {\color{lightgray}0} & {\color{lightgray}0}\\
{\color{lightgray}0} & {\color{lightgray}0} & {\color{lightgray}0} & 1 & {\color{lightgray}0} & {\color{lightgray}0} & {\color{lightgray}0} & {\color{lightgray}0} & {\color{lightgray}0} & {\color{lightgray}0}\\
{\color{lightgray}0} & {\color{lightgray}0} & {\color{lightgray}0} & {\color{lightgray}0} & 1 & {\color{lightgray}0} & {\color{lightgray}0} & {\color{lightgray}0} & {\color{lightgray}0} & {\color{lightgray}0}\\
{\color{lightgray}0} & {\color{lightgray}0} & {\color{lightgray}0} & {\color{lightgray}0} & {\color{lightgray}0} & {\color{lightgray}0} & {\color{lightgray}0} & 1 & {\color{lightgray}0} & {\color{lightgray}0}\\
{\color{lightgray}0} & {\color{lightgray}0} & {\color{lightgray}0} & {\color{lightgray}0} & {\color{lightgray}0} & {\color{lightgray}0} & {\color{lightgray}0} & {\color{lightgray}0} & 1 & {\color{lightgray}0}\\
{\color{lightgray}0} & {\color{lightgray}0} & {\color{lightgray}0} & {\color{lightgray}0} & {\color{lightgray}0} & {\color{lightgray}0} & {\color{lightgray}0} & {\color{lightgray}0} & {\color{lightgray}0} & 1\\
{\color{lightgray}0} & {\color{lightgray}0} & {\color{lightgray}0} & {\color{lightgray}0} & {\color{lightgray}0} & 1 & {\color{lightgray}0} & {\color{lightgray}0} & {\color{lightgray}0} & {\color{lightgray}0}\\
{\color{lightgray}0} & {\color{lightgray}0} & {\color{lightgray}0} & {\color{lightgray}0} & {\color{lightgray}0} & {\color{lightgray}0} & 1 & {\color{lightgray}0} & {\color{lightgray}0} & {\color{lightgray}0}
\end{smallmatrix}\right).
\]
 Using these matrices we can write:
\begin{eqnarray*}
A^{\prime}+\lambda B^{\prime} & = & F_{2}^{-1}F_{2}(A^{\prime}+\lambda B^{\prime})=F_{2}^{-1}(L_{1}+\lambda L_{2})F_{1}\\
 & = & F_{2}^{-1}(L_{1}+\lambda L_{2})G_{1}G_{1}^{-1}F_{1}=F_{2}^{-1}G_{2}(A+\lambda B)G_{1}^{-1}F_{1}\\
 & = & D_{1}\begin{pmatrix}\mathbb I_{8} & 0\end{pmatrix}D_{2}(A+\lambda B)C_{1}\begin{pmatrix}\mathbb I_{12}\\
0
\end{pmatrix}C_{2}.
\end{eqnarray*}
 So $D_{1}^{-1}(A^{\prime}+\lambda B^{\prime})C_{2}^{-1}=\begin{pmatrix}\mathbb I_{8} & 0\end{pmatrix}D_{2}(A+\lambda B)C_{1}\begin{pmatrix}\mathbb I_{12}\\
0
\end{pmatrix}$, hence
\begin{eqnarray*}
A^{\prime}+\lambda B^{\prime} & = & \begin{pmatrix}\mathbb I_{8} & 0\end{pmatrix}\begin{pmatrix}D_{1}\\
 & \mathbb I_{2}
\end{pmatrix}D_{2}(A+\lambda B)C_{1}\begin{pmatrix}C_{2}\\
 & \mathbb I_{2}
\end{pmatrix}\begin{pmatrix}\mathbb I_{12}\\
0
\end{pmatrix}\\
 & = & \begin{pmatrix}\mathbb I_{8} & 0\end{pmatrix}D_{2}(A+\lambda B)C^{\prime}\begin{pmatrix}\mathbb I_{12}\\
0
\end{pmatrix},
\end{eqnarray*}
 where
\[
C^{\prime}=\left(\begin{smallmatrix}{\color{lightgray}0} & 1 & {\color{lightgray}0} & {\color{lightgray}0} & {\color{lightgray}0} & {\color{lightgray}0} & {\color{lightgray}0} & {\color{lightgray}0} & {\color{lightgray}0} & {\color{lightgray}0} & {\color{lightgray}0} & {\color{lightgray}0} & {\color{lightgray}0} & {\color{lightgray}0}\\
{\color{lightgray}0} & {\color{lightgray}0} & 1 & {\color{lightgray}0} & {\color{lightgray}0} & {\color{lightgray}0} & {\color{lightgray}0} & {\color{lightgray}0} & {\color{lightgray}0} & {\color{lightgray}0} & {\color{lightgray}0} & {\color{lightgray}0} & {\color{lightgray}0} & {\color{lightgray}0}\\
{\color{lightgray}0} & {\color{lightgray}0} & {\color{lightgray}0} & 1 & {\color{lightgray}0} & {\color{lightgray}0} & {\color{lightgray}0} & {\color{lightgray}0} & {\color{lightgray}0} & {\color{lightgray}0} & {\color{lightgray}0} & {\color{lightgray}0} & {\color{lightgray}0} & {\color{lightgray}0}\\
{\color{lightgray}0} & {\color{lightgray}0} & {\color{lightgray}0} & {\color{lightgray}0} & 1 & {\color{lightgray}0} & {\color{lightgray}0} & {\color{lightgray}0} & {\color{lightgray}0} & {\color{lightgray}0} & {\color{lightgray}0} & {\color{lightgray}0} & 1 & {\color{lightgray}0}\\
{\color{lightgray}0} & {\color{lightgray}0} & {\color{lightgray}0} & {\color{lightgray}0} & 1 & {\color{lightgray}0} & {\color{lightgray}0} & {\color{lightgray}0} & {\color{lightgray}0} & {\color{lightgray}0} & {\color{lightgray}0} & {\color{lightgray}0} & {\color{lightgray}0} & 1\\
{\color{lightgray}0} & {\color{lightgray}0} & {\color{lightgray}0} & {\color{lightgray}0} & {\color{lightgray}0} & 1 & {\color{lightgray}0} & {\color{lightgray}0} & {\color{lightgray}0} & {\color{lightgray}0} & {\color{lightgray}0} & {\color{lightgray}0} & {\color{lightgray}0} & {\color{lightgray}0}\\
{\color{lightgray}0} & {\color{lightgray}0} & {\color{lightgray}0} & {\color{lightgray}0} & {\color{lightgray}0} & {\color{lightgray}0} & 1 & {\color{lightgray}0} & {\color{lightgray}0} & {\color{lightgray}0} & {\color{lightgray}0} & {\color{lightgray}0} & {\color{lightgray}0} & {\color{lightgray}0}\\
{\color{lightgray}0} & {\color{lightgray}0} & {\color{lightgray}0} & {\color{lightgray}0} & 1 & {\color{lightgray}0} & {\color{lightgray}0} & {\color{lightgray}0} & {\color{lightgray}0} & {\color{lightgray}0} & {\color{lightgray}0} & {\color{lightgray}0} & {\color{lightgray}0} & {\color{lightgray}0}\\
1 & {\color{lightgray}0} & {\color{lightgray}0} & {\color{lightgray}0} & {\color{lightgray}0} & {\color{lightgray}0} & {\color{lightgray}0} & {\color{lightgray}0} & {\color{lightgray}0} & {\color{lightgray}0} & {\color{lightgray}0} & {\color{lightgray}0} & {\color{lightgray}0} & {\color{lightgray}0}\\
{\color{lightgray}0} & {\color{lightgray}0} & {\color{lightgray}0} & {\color{lightgray}0} & {\color{lightgray}0} & {\color{lightgray}0} & {\color{lightgray}0} & 1 & {\color{lightgray}0} & {\color{lightgray}0} & {\color{lightgray}0} & {\color{lightgray}0} & {\color{lightgray}0} & {\color{lightgray}0}\\
{\color{lightgray}0} & {\color{lightgray}0} & {\color{lightgray}0} & {\color{lightgray}0} & {\color{lightgray}0} & {\color{lightgray}0} & {\color{lightgray}0} & {\color{lightgray}0} & 1 & {\color{lightgray}0} & {\color{lightgray}0} & {\color{lightgray}0} & {\color{lightgray}0} & {\color{lightgray}0}\\
{\color{lightgray}0} & {\color{lightgray}0} & {\color{lightgray}0} & {\color{lightgray}0} & {\color{lightgray}0} & {\color{lightgray}0} & {\color{lightgray}0} & {\color{lightgray}0} & {\color{lightgray}0} & 1 & {\color{lightgray}0} & {\color{lightgray}0} & {\color{lightgray}0} & {\color{lightgray}0}\\
{\color{lightgray}0} & {\color{lightgray}0} & {\color{lightgray}0} & {\color{lightgray}0} & {\color{lightgray}0} & {\color{lightgray}0} & {\color{lightgray}0} & {\color{lightgray}0} & {\color{lightgray}0} & {\color{lightgray}0} & 1 & {\color{lightgray}0} & {\color{lightgray}0} & {\color{lightgray}0}\\
{\color{lightgray}0} & {\color{lightgray}0} & {\color{lightgray}0} & {\color{lightgray}0} & {\color{lightgray}0} & {\color{lightgray}0} & {\color{lightgray}0} & {\color{lightgray}0} & {\color{lightgray}0} & {\color{lightgray}0} & {\color{lightgray}0} & 1 & {\color{lightgray}0} & {\color{lightgray}0}
\end{smallmatrix}\right).
\]
 Obviously, $A+\lambda B\sim D_{2}(A+\lambda B)C^{\prime}$, where
\[
D_{2}(A+\lambda B)C^{\prime}=\left(\begin{smallmatrix}0 & \lambda & 1 & 0 & 0 & 0 & 0 &  &  &  &  &  & {\color{lightgray}0} & {\color{lightgray}0}\\
 & 0 & \lambda & 1 & 0 & 0 & 0 &  &  &  &  &  & {\color{lightgray}0} & {\color{lightgray}0}\\
 & 0 & 0 & \lambda & 1 & 0 & 0 &  &  &  &  &  & 1 & {\color{lightgray}0}\\
 & 0 & 0 & 0 & \lambda & 1 & 0 &  &  &  &  &  & {\color{lightgray}0} & \lambda\\
 & 0 & 0 & 0 & 0 & \lambda & 1 &  &  &  &  &  & {\color{lightgray}0} & {\color{lightgray}0}\\
 &  &  &  &  &  &  & \lambda & 1 & 0 &  &  & {\color{lightgray}0} & {\color{lightgray}0}\\
 &  &  &  &  &  &  & 0 & \lambda & 1 &  &  & {\color{lightgray}0} & {\color{lightgray}0}\\
 &  &  &  &  &  &  &  &  &  & \lambda & 1 & {\color{lightgray}0} & {\color{lightgray}0}\\
{\color{lightgray}0} & {\color{lightgray}0} & {\color{lightgray}0} & {\color{lightgray}0} & 1 & {\color{lightgray}0} & \lambda & {\color{lightgray}0} & {\color{lightgray}0} & {\color{lightgray}0} & {\color{lightgray}0} & {\color{lightgray}0} & {\color{lightgray}0} & {\color{lightgray}0}\\
\lambda & {\color{lightgray}0} & {\color{lightgray}0} & {\color{lightgray}0} & {\color{lightgray}0} & {\color{lightgray}0} & {\color{lightgray}0} & 1 & {\color{lightgray}0} & {\color{lightgray}0} & {\color{lightgray}0} & {\color{lightgray}0} & {\color{lightgray}0} & {\color{lightgray}0}
\end{smallmatrix}\right)=\left(\begin{array}{cc}
A^{\prime}+\lambda B^{\prime} & A_{12}+\lambda B_{12}\\
A_{21}+\lambda B_{21} & A_{22}+\lambda B_{22}
\end{array}\right),
\]
 with the completion pencils
\[
A_{12}+\lambda B_{12}=\left(\begin{smallmatrix}{\color{lightgray}0} & {\color{lightgray}0}\\
{\color{lightgray}0} & {\color{lightgray}0}\\
1 & {\color{lightgray}0}\\
{\color{lightgray}0} & \lambda\\
{\color{lightgray}0} & {\color{lightgray}0}\\
{\color{lightgray}0} & {\color{lightgray}0}\\
{\color{lightgray}0} & {\color{lightgray}0}\\
{\color{lightgray}0} & {\color{lightgray}0}
\end{smallmatrix}\right),\quad A_{21}+\lambda B_{21}=\left(\begin{smallmatrix}{\color{lightgray}0} & {\color{lightgray}0} & {\color{lightgray}0} & {\color{lightgray}0} & 1 & {\color{lightgray}0} & \lambda & {\color{lightgray}0} & {\color{lightgray}0} & {\color{lightgray}0} & {\color{lightgray}0} & {\color{lightgray}0}\\
\lambda & {\color{lightgray}0} & {\color{lightgray}0} & {\color{lightgray}0} & {\color{lightgray}0} & {\color{lightgray}0} & {\color{lightgray}0} & 1 & {\color{lightgray}0} & {\color{lightgray}0} & {\color{lightgray}0} & {\color{lightgray}0}
\end{smallmatrix}\right),\quad A_{22}+\lambda B_{22}=\left(\begin{smallmatrix}{\color{lightgray}0} & {\color{lightgray}0}\\
{\color{lightgray}0} & {\color{lightgray}0}
\end{smallmatrix}\right).
\]
\end{example}
\begin{rem}
The calculations were verified using the computer algebra system Maxima
\cite{MAX}.\end{rem}
\bigskip

{\it Acknowledgements.} This work was supported by the Bolyai Scholarship of the Hungarian Academy of Sciences and Grant PN
II-RU-TE-2009-1-ID 303.

\end{document}